\documentclass[a4paper,12pt,reqno]{amsart}
\usepackage{amssymb,amsmath,hyperref}
\title{Splitting Appell functions in terms of single quotients of theta functions}
\author{Eric T. Mortenson}
\address{Department of Mathematics and Computer Science, Saint Petersburg State University, Saint Petersburg, 199034, Russia}
\email{etmortenson@gmail.com}
\author{Dilshod Urazov}
\address{Department of Mathematics and Computer Science, Saint Petersburg State University, Saint Petersburg, 199034, Russia}
\email{urazofficial@gmail.com}

\renewcommand\theta{\vartheta}

\newtheorem{theorem}{Theorem}

\newtheorem{corollary}[theorem]{Corollary}
\newtheorem{proposition}[theorem]{Proposition}
\theoremstyle{definition}

\numberwithin{theorem}{section} 
\numberwithin{equation}{section}

\makeatletter
\@namedef{subjclassname@2020}{\textup{2020} Mathematics Subject Classification}
\makeatother

\begin{document}

\date{15 May 2023}



\begin{abstract}
    Ramanujan's last letter to Hardy introduced the world to mock theta functions, and the mock theta function identities found in Ramanujan's lost notebook added to their intriguing nature.  For example, we find the four tenth-order mock theta functions and their six identities.  The six identities themselves are of a spectacular nature and were first proved by Choi.  We also find over eight sixth-order mock theta functions in the lost notebook, but among their many identities there is only one relationship like those of the tenth-orders.  Recently, three new identities for the sixth-order mock theta functions that are in the spirit of the six tenth-order identities were discovered.   Here we present several families of tenth-order like identities for Appell functions, which are the building blocks of Ramanujan's mock theta functions.
\end{abstract}

\maketitle
%\setcounter{section}{-1}
%\tableofcontents

\section{Introduction}

We have many themes in the introduction.  We have Ramanujan's last letter to Hardy, Ramanujan's Lost Notebook as discovered by George Andrews, mock theta functions, mock theta function identities involving single quotients of theta functions, various notions of building blocks of mock theta functions, properties of the building blocks, and mock theta function identities that could have easily been in the lost notebook but are mysteriously absent.  The last theme helps to fuel speculation that pages from the lost notebook are missing \cite[p. 287]{ABV}.

In Ramanujan's last letter to Hardy, he gave a list of seventeen functions which he called ``mock theta functions.''   Each function was defined by Ramanujan as a $q$-series convergent for $|q|<1$.  He stated that they have certain asymptotic properties similar to those of ordinary theta functions, but that they are not theta functions.  He also stated identities relating some of the mock theta functions to each other.   One finds four `3rd' order mock theta functions and several identities; ten `5th' order functions and identities; and three `7th' order functions, but with the statement that they are not related.   

Later, many more mock theta function identities were found in the lost notebook.   The newly discovered identities included the mock theta conjectures, which were ten identities where each identity expressed a different 5th order mock theta function in terms of a building block and a single quotient of theta functions.  This particular building block is the so-called universal mock theta function $g(x;q)$.

Let $q:=q_{\tau}=e^{2 \pi i \tau}$, $\tau\in\mathbb{H}:=\{ z\in \mathbb{C}| \textup{Im}(z)>0 \}$, and define $\mathbb{C}^*:=\mathbb{C}-\{0\}$.  Recall
\begin{gather*}
(x)_n=(x;q)_n:=\prod_{i=0}^{n-1}(1-q^ix), \ \ (x)_{\infty}=(x;q)_{\infty}:=\prod_{i\ge 0}(1-q^ix).
\end{gather*}
We define the universal mock theta function $g(x;q)$ as
\begin{equation*}
g(x;q):=x^{-1}\Big ( -1 +\sum_{n=0}^{\infty}\frac{q^{n^2}}{(x)_{n+1}(q/x)_{n}} \Big ).
\end{equation*}
Furthermore, we define the theta function
\begin{equation*}
 \Theta(x;q):=(x)_{\infty}(q/x)_{\infty}(q)_{\infty}=\sum_{n=-\infty}^{\infty}(-1)^nq^{\binom{n}{2}}x^n,
\end{equation*}
where we also use the abbreviations
\begin{gather*}
\Theta_{a,m}:=\Theta(q^a;q^m), \ \ \Theta_m:=\Theta_{m,3m}=\prod_{i\ge 1}(1-q^{mi}), \ {\text{and }}\overline{\Theta}_{a,m}:=\Theta(-q^a;q^m).
\end{gather*}

Two of the ten mock theta conjectures then read
\begin{align*}
f_0(q)&:=\sum_{n= 0}^{\infty}\frac{q^{n^2}}{(-q)_n}=\frac{\Theta_{5,10}\Theta_{2,5}}{\Theta_1}-2q^2g(q^2;q^{10}),\\
f_1(q)&:=\sum_{n= 0}^{\infty}\frac{q^{n(n+1)}}{(-q)_n}=\frac{\Theta_{5,10}\Theta_{1,5}}{J_1}-2q^3g(q^4;q^{10}).
\end{align*}
In the course of proving the mock theta conjectures \cite{AG,H1}, Hickerson found mock theta conjecture analogs for the three 7th order functions \cite{H2}, one of which reads
\begin{equation*}
{\mathcal{F}}_0(q):=\sum_{n\ge 0}\frac{q^{n^2}}{(q^{n+1};q)_{n}}=2+2qg(q;q^{7})-\frac{\Theta_{3,7}^2}{\Theta_1}.
\end{equation*}

The universal mock theta function $g(x;q)$ is only related to the odd ordered mock theta functions.  Although the three seventh-order mock theta functions are introduced in the last letter, they do not appear in the lost notebook.  

Also in the lost notebook, we find the four tenth-order mock theta functions and their six identities \cite{RLN}.  The six identities themselves are of a fantastic nature and were first proved by Choi \cite{C1, C2, C3} using methods similar to those of Hickerson \cite{H1, H2}.  We also find over eight sixth-order mock theta functions in the lost notebook, but among their many identities there is only one relationship like those of the tenth-orders \cite[p. 135, Entry 7.4.2]{ABV}, \cite{RLN}.  Recently, three more tenth-order like identities were discovered for the sixth order mock theta functions \cite{Mo2022}.

We recall the four tenth-order mock theta functions \cite{C1, C2, C3, RLN}
{\allowdisplaybreaks \begin{align*}
{\phi}_{10}(q)&:=\sum_{n\ge 0}\frac{q^{\binom{n+1}{2}}}{(q;q^2)_{n+1}}, \ \ {\psi}_{10}(q):=\sum_{n\ge 0}\frac{q^{\binom{n+2}{2}}}{(q;q^2)_{n+1}}, \\\ 
& \ \ \ \ \ {X}_{10}(q):=\sum_{n\ge 0}\frac{(-1)^nq^{n^2}}{(-q;q)_{2n}}, \ \  {\chi}_{10}(q):=\sum_{n\ge 0}\frac{(-1)^nq^{(n+1)^2}}{(-q;q)_{2n+1}}.\notag
\end{align*}}%
The four functions satisfy six identities, which we have slightly-rewritten in order to emphasize the single quotient of theta functions.  Letting $\omega$ be a primitive third root of unity, we have \cite{C1,C2, RLN}
{\allowdisplaybreaks \begin{align}
q^{2}\phi_{10}(q^9)-\frac{\psi_{10}(\omega q)-\psi_{10}(\omega^2 q)}{\omega - \omega^2}
&=-q\frac{\Theta_{1,2}}{\Theta_{3,6}}\cdot \frac{\Theta_{3,15}\Theta_{6}}{\Theta_{3}},
\label{equation:tenth-id-1}\\
q^{-2}\psi_{10}(q^9)+\frac{\omega \phi_{10}(\omega q)-\omega^2\phi_{10}(\omega^2 q)}{\omega - \omega^2}
&=\frac{\Theta_{1,2}}{\Theta_{3,6}}\cdot \frac{\Theta_{6,15}\Theta_{6}}{\Theta_{3}},
\label{equation:tenth-id-2}\\
X_{10}(q^9)-\frac{\omega \chi_{10}(\omega q)-\omega^2\chi_{10}(\omega^2 q)}{\omega - \omega^2}
&=\frac{\overline{\Theta}_{1,4}}{\overline{\Theta}_{3,12}}\cdot \frac{\Theta_{18,30}\Theta_{3}}{\Theta_{6}},
\label{equation:tenth-id-3}\\
\chi_{10}(q^9)+q^{2}\frac{ X_{10}(\omega q)-X_{10}(\omega^2 q)}{\omega - \omega^2}&=-q^3\frac{\overline{\Theta}_{1,4}}{\overline{\Theta}_{3,12}}\cdot 
\frac{\Theta_{6,30}\Theta_{3}}{\Theta_{6}},
\label{equation:tenth-id-4}
\end{align}}%
and \cite{C3, RLN}
\begin{align}
\phi_{10}(q)-q^{-1}\psi_{10}(-q^4)+q^{-2}\chi_{10}(q^8)&=\frac{\overline{\Theta}_{1,2}\Theta(-q^2;-q^{10})}{\Theta_{2,8}},
\label{equation:tenth-id-5}\\
\psi_{10}(q)+q\phi_{10}(-q^4)+X_{10}(q^8)&=\frac{\overline{\Theta}_{1,2}\Theta(-q^6;-q^{10})}{\Theta_{2,8}}.
\label{equation:tenth-id-6}
\end{align}

Zwegers \cite{Zw3} has since given short proof of identities (\ref{equation:tenth-id-1})--(\ref{equation:tenth-id-4}).  Mortenson \cite{Mo2018} later gave short proofs of all six identities (\ref{equation:tenth-id-1})--(\ref{equation:tenth-id-6}).

\smallskip
What led Ramanujan to these six identities is a continuing mystery.  Indeed, in Andrews and Berndt's fifth volume on Ramanujan's lost notebook \cite[p. 396]{ABV}, they state 

\smallskip
``{\em It is inconceivable that an identity such as (\ref{equation:tenth-id-5}) could be stumbled upon by a mindless search algorithm without any overarching theoretical insight.}''   

\smallskip
This brings us to a notion of a building block which is in a sense finer than that of the universal mock theta function $g(x;q)$.  Appell functions are the building blocks for both even and odd ordered mock theta functions \cite[Section 5]{HM}, \cite{Zw}.  We will define them as follows
\begin{equation}
m(x,z;q):=\frac{1}{\Theta(z;q)}\sum_{r=-\infty}^{\infty}\frac{(-1)^rq^{\binom{r}{2}}z^r}{1-q^{r-1}xz}.\label{equation:mdef-eq}
\end{equation}
The universal mock theta function can be expressed in terms of Appell functions \cite[Proposition $4.2$]{HM}, \cite[Theorem $2.2$]{H1}:
\begin{equation*}
g(x;q)=-x^{-1}m(q^2x^{-3},x^2;q^3)-x^{-2}m(qx^{-3},x^2;q^3).
\end{equation*}

In \cite{Mo2018}, Mortenson gave short proofs of all six of Ramanujan's identities for the tenth-order mock theta functions by using a recent result on Appell function properties.  

\begin{theorem} \label{theorem:msplit-general-n} \cite[Theorem $3.5$]{HM} For generic $x,z,z'\in \mathbb{C}^*$ 
{\allowdisplaybreaks \begin{align}
 D_n(x,z,z';q)=z' \Theta_n^3  \sum_{r=0}^{n-1}
\frac{q^{{\binom{r}{2}}} (-xz)^r
\Theta\big(-q^{{\binom{n}{2}+r}} (-x)^n z z';q^n\big)
\Theta(q^{nr} z^n/z';q^{n^2})}
{\Theta(xz;q) \Theta(z';q^{n^2}) \Theta\big(-q^{{\binom{n}{2}}} (-x)^n z';q^n)\Theta(q^r z;q^n\big )},
\end{align}}%
where
\begin{equation}
D_n(x,z,z';q):=m(x,z;q) - \sum_{r=0}^{n-1} q^{{-\binom{r+1}{2}}} (-x)^r m\big({-}q^{{\binom{n}{2}-nr}} (-x)^n, z'; q^{n^2} \big).
\label{equation:Dn-def}
\end{equation}
\end{theorem}
\noindent  If one lets $n$ be an odd prime number, then one can think of \cite[Theorem $3.5$]{HM} as splitting the Appell function to determine on which arithmetic progressions of Fourier coefficients the mockness lives.  This was the idea behind an elementary proof \cite{HM2} of celebrated results of Bringmann and Ono \cite{BrO2}. There is also a counterpart to \cite[Theorem $3.5$]{HM} which involves summing over roots of unity \cite[Theorem $3.9$]{HM}, but we will not need it here.

 The idea behind the proofs in \cite{Mo2018} is straightforward.  Once one has the Appell function forms of the tenth-order mock theta functions, one regroups the Appell functions by using (\ref{equation:Dn-def}) and then replaces them with the appropriate sums of quotients of theta functions given by Theorem \ref{theorem:msplit-general-n}.  Each of the six identities is then reduced to proving a theta function identity which can be verified through repeated use of the three-term Weierstrass relation for theta functions \cite[(1.)]{We}, \cite{Ko}, see (\ref{equation:Weierstrass}).

In terms of Appell functions, two of Ramanujan's sixth-order mock theta functions  \cite{AH}, \cite[Section 5]{HM} read
\begin{align*}
\phi(q)&:=\sum_{n\ge 0}\frac{(-1)^nq^{n^2}(q;q^2)_n}{(-q)_{2n}}=2m(q,-1;q^3),\\
\psi(q)&:=\sum_{n\ge 0}\frac{(-1)^nq^{(n+1)^2}(q;q^2)_n}{(-q)_{2n+1}}=m(1,-q;q^3).
\end{align*}

\noindent There are numerous identities for the sixth-order functions in the lost notebook \cite{AH, BC}, but there is only one similar to the tenth-order identities (\ref{equation:tenth-id-1})-(\ref{equation:tenth-id-6}), see \cite[p. 135, Entry 7.4.2]{ABV}, \cite[p. 13, equation 5b]{RLN}
\begin{equation}
\phi(q^9)-\psi(q)-q^{-3}\psi(q^9)=\frac{\overline{\Theta}_{3,12}\Theta_{6}^2}{\overline{\Theta}_{1,4}\overline{\Theta}_{9,36}}
\label{equation:RLN6-A}.
\end{equation}

Of course, there are many more sixth-order mock theta functions in the lost notebook \cite{AH, BC}.   It is natural to ask if they too enjoy identities similar to (\ref{equation:RLN6-A}).   Surprisingly this is true, and the identities were only recently discovered.

\begin{theorem}\cite[Theorem 1.2]{Mo2022} \label{theorem:main} The following identities for the sixth-order mock theta functions $\rho(q)$, $\sigma(q)$, $\lambda(q)$, $\mu(q)$, $\phi_{\_}(q)$, and $\psi_{\_}(q)$ are true
\begin{align}
q\rho(q)+q^3\rho(q^9)-2\sigma(q^9)
&=q\frac{\Theta_{3,6}\Theta_{3}^2}{\Theta_{1,2}\Theta_{9,18}},
\label{equation:newSixth-1}\\
q\lambda(q)+q^{3}\lambda(q^9)-2\mu(q^{9})
&=-\frac{\Theta_{3,6}\Theta_{6}^2}{\overline{\Theta}_{1,4}\overline{\Theta}_{9,36}},
\label{equation:newSixth-2}\\
\psi\_(q)+q^{-3}\psi\_(q^9)-\phi\_(q^9)
&=q\frac{\overline{\Theta}_{3,12}\Theta_{3}^2}{\Theta_{1,2}\Theta_{9,18}},
\label{equation:newSixth-3}
\end{align}
where the sixth-order mock theta functions  $\rho(q)$, $\sigma(q)$, $\lambda(q)$, $\mu(q)$, $\phi_{\_}(q)$, and $\psi_{\_}(q)$ are all found in Ramanujan's lost notebook.
\end{theorem}

Although properties such as Theorem \ref{theorem:msplit-general-n} are not explicitly present in the lost notebook, there are hints.  The following identity for $g(x;q)$ can be found in the lost notebook \cite[p. $32$]{RLN}, \cite[$(12.5.3)$]{ABI} and resembles a special case of Theorem \ref{theorem:msplit-general-n}:
\begin{equation*}
g(x;q)=-x^{-1}+qx^{-3}g(-qx^{-2};q^4)-qg(-qx^2;q^4)+\frac{(q^2;q^2)_{\infty}^5}{x(q^4;q^4)_{\infty}^2\Theta(x;q)\Theta(-qx^2;q^2)}.
\end{equation*}
There are also identities for $g(x;q)$ in the lost notebook that resemble special cases of \cite[Theorem $3.9$]{HM} where one sums Appell functions over roots of unity, see  \cite[p. $39$]{RLN}, \cite[$(12.4.4)$]{ABI}:
\begin{equation*}
g(x;q)+g(-x;q)=-2qg(-qx^2;q^4)
+\frac{2(q^2;q^2)_{\infty}^5}{(q)_{\infty}^2\Theta(-qx^2;q^4)\Theta(x^2;q^2)}.
\end{equation*}
We note that both identities for $g(x;q)$ involve a single quotient of theta functions.

What we want to do here is to find families of specializations of Theorem \ref{theorem:msplit-general-n} which yield single-quotients of theta functions for $n=2$, $3$, and $4$.   We already have two examples, but would like to find more.  From \cite[Corollaries $3.7$, $3.8$]{HM}, we have
\begin{gather*}
D_{2}(x,z,z^4;q)=D_2(x,x^{-1}z^{-1},z^4;q)=-\frac{(q^2;q^2)_{\infty}(q^4;q^4)_{\infty}\Theta(-xz^2;q)\Theta(-xz^3;q)}{x\Theta(xz;q)\Theta(z^4;q^4)\Theta(-qx^2z^4;q^2)},\\
D_{3}(x,-1,-1;q)=D_2(x,x^{-1},-1;q)=\frac{x(q;q)_{\infty}(q^3;q^3)_{\infty}^2(q^6;q^6)_{\infty}(q^9;q^9)_{\infty}\Theta(qx^2;q^2)}{2q(q^2;q^2)_{\infty}^2(q^{18};q^{18})_{\infty}^2\Theta(-x^3;q^3)}.
\end{gather*}

  In Sections \ref{section:N2}, \ref{section:N3}, and \ref{section:N4}, we state families of specializations of Theorem \ref{theorem:msplit-general-n} for $n=2$, $3$ and $4$, that yield single quotients of theta functions.  In Section \ref{section:prelim}, we recall necessary facts about theta functions and develop a few auxiliary identities.   In Section \ref{section:N2-proofs}, \ref{section:N3-proofs}, and \ref{section:N4-proofs} respectively we prove the identities. In Section \ref{section:ThetaIds} we list theta function identities that are byproducts of the proofs found in Sections \ref{section:N3-proofs} and \ref{section:N4-proofs}.  In Section \ref{section:ThetaIds-proofs}, we prove the new theta function identities.

\section*{Acknowledgements}
We would like to thank Dean Hickerson for his helpful comments and suggestions.  This work was supported by the Theoretical Physics and Mathematics Advancement Foundation BASIS, agreement No. 20-7-1-25-1.

\section{Statement of results: the case $n=2$}\label{section:N2}

We point out the $n=2$ specialization of \cite[Theorem $3.5$]{HM} and then present a list of families of specializations where each family evaluates to a single quotient of theta functions.
\begin{corollary} \label{corollary:msplitn2zprime} For generic $x,z,z'\in \mathbb{C}^*$ 
{\allowdisplaybreaks \begin{align}
D_2&(x,z,z';q) \label{equation:msplit2}\\
&=\frac{z'(q^2;q^2)_{\infty}^3}{\Theta(xz;q)\Theta(z';q^4)}\Big [
\frac{\Theta(-qx^2zz';q^2)\Theta(z^2/z';q^{4})}{\Theta(-qx^2z';q^2)\Theta(z;q^2)}
-xz \frac{\Theta(-q^2x^2zz';q^2)\Theta(q^2z^2/z';q^{4})}{\Theta(-qx^2z';q^2)\Theta(qz;q^2)}\Big ],\notag
\end{align}}%
where
\begin{equation}
D_2(x,z,z';q):=m(x,z;q)-m(-qx^2,z';q^4 )+q^{-1}xm(-q^{-1}x^2,z';q^4).\label{equation:D2-def}
\end{equation}
\end{corollary}

We have the following list of families of specializations that yield single quotients of theta functions.
\begin{theorem} \label{theorem:caseN2} We have
{\allowdisplaybreaks \begin{gather}
D_2(x,z,z^2;q)=-\frac{xz^3(q^2;q^2)_{\infty}^3\Theta(-q^2x^2z^3;q^2)\Theta(q^2;q^{4})}
{\Theta(xz;q)\Theta(z^2;q^4)\Theta(-qx^2z^2;q^2)\Theta(qz;q^2)}, 
\label{equation:idN2-1}\\
D_2(x,z,z^4;q)=-\frac{(q^2;q^2)_{\infty}(q^4;q^4)_{\infty}\Theta(-xz^2;q)\Theta(-xz^3;q)}{x\Theta(xz;q)\Theta(z^4;q^4)\Theta(-qx^2z^4;q^2)},
\label{equation:idN2-2}\\
D_2(x,z,x^{-1};q)=-\frac{z(q;q)_{\infty}^3\Theta(-qxz^2;q^2)}
{\Theta(xz;q)\Theta(-qx;q^2)\Theta(z;q)},
\label{equation:idN2-3}\\
D_2(x,z,x^{-2};q)=-\frac{(q;q)_{\infty}^3\Theta(-xz^2;q^2)}{\Theta(xz;q)\Theta(-x;q^2)\Theta(z;q)},
\label{equation:idN2-4}\\
D_2(u^3,z,u^{-4};q)=
-\frac{(q^2;q^2)_{\infty}(q^4;q^4)_{\infty}\Theta(-u^2z;q) \Theta(u;q)\Theta(-zu;q)}
{\Theta(u^3z;q)\Theta(u^{4};q^4)\Theta(-qu^2;q^2)\Theta(z;q)},
\label{equation:idN2-5}\\
D_2(qz^{-2},z,z^3;q^3)=-\frac{q^{-1}z^2(q^3;q^3)_{\infty}^4(q^6;q^6)_{\infty}\Theta(z;q^4)}
{(q;q)_{\infty}\Theta(q^2z;q^3)\Theta(z^3;q^{12})\Theta(-qz;q^6)\Theta(z;q^3)},
\label{equation:idN2-6}\\
D_2(q^2z^{-2},z,z^3;q^3)=- \frac{q^{-1}z(q^3;q^3)_{\infty}^4(q^6;q^6)_{\infty}\Theta(z;q^4)}
{(q;q)_{\infty}\Theta(qz;q^3)\Theta(z^3;q^{12})\Theta(-q^5z;q^6)\Theta(z;q^3)},
\label{equation:idN2-7}\\
D_2(x,z,-qx^{-2}z^{-2};q)=-\frac{x^{-1}(q^2;q^2)_{\infty}^3\Theta(-xz^2;q)}
  {\Theta(xz;q)\Theta(-q^3x^{2}z^{2};q^4)\Theta(z^{2};q^2)}.
 \label{equation:idN2-8}
\end{gather}}%
\end{theorem}

\section{Statement of results: the case $n=3$}\label{section:N3}
We point out the $n=3$ specialization of \cite[Theorem $3.5$]{HM} and then present a list of families of specializations where each family evaluates to a single quotient of theta functions.

\begin{corollary} \label{corollary:msplitn3zprime} For generic $x,z,z'\in \mathbb{C}^*$ 
\begin{align}
D_3(x,z,z';q)&=\frac{z'(q^3;q^3)_{\infty}^3}{\Theta(xz;q)\Theta(z';q^{9})\Theta(x^3z';q^3)}\Big [ 
\frac{1}{z}\frac{\Theta(x^3zz';q^3)\Theta(z^3/z';q^{9})}{\Theta(z;q^3)}\label{equation:msplit3} \\
&\ \ \ \ \ -\frac{x}{q}\frac{\Theta(qx^3zz';q^3)\Theta(q^{3}z^3/z';q^{9})}{\Theta(qz;q^3)}
+\frac{x^2z}{q}\frac{\Theta(q^2x^3zz';q^3)\Theta(q^{6}z^3/z';q^{9})}{\Theta(q^2z;q^3)}\Big ],\notag
\end{align}
where
\begin{align}
D_3(x,z,z';q)&:=m(x,z;q)-m (q^{3}x^3,z';q^{9} )\label{equation:D3-def}\\
&\ \ \ \ \  +q^{-1}xm (x^3,z';q^{9} )-q^{-3}x^2m (q^{-3}x^3,z';q^{9} ).\notag
\end{align}
\end{corollary}

We have the following list of families of specializations that yield single quotients of theta functions.

\newpage
\begin{theorem}\label{theorem:caseN3} We have
 \begin{align}
D_{3}(z^{-4},z,q^3z^9;q)
&=D_{3}(z^{-4},z^3,q^3z^9;q)\label{equation:idN3-1}\\
&=\frac{z^5(q;q)_{\infty}(q^3;q^3)_{\infty}^4}{\Theta(z;q)\Theta(z^3;q^3)\Theta(q^2z^3;q^3)\Theta(q^3z^9;q^9)},
\notag\\
D_{3}(z^{-4},z,q^6z^9;q)
&=D_{3}(z^{-4},z^3,q^6z^9;q)\label{equation:idN3-2}\\
&=\frac{z^3(q;q)_{\infty}(q^3;q^3)_{\infty}^4}{\Theta(z;q)\Theta(z^3;q^3)\Theta(qz^3;q^3)\Theta(q^6z^9;q^9)},
\notag \\
D_{3}(u^{-5},u^2,q^3u^9;q)
&=D_{3}(u^{-5},u^3,q^3u^9;q)\label{equation:idN3-3}\\
&=\frac{u^7(q;q)_{\infty}(q^3;q^3)_{\infty}^4}{\Theta(u;q)\Theta(u^6;q^3)\Theta(q^2u^3;q^3)\Theta(q^3u^9;q^9)},
\notag \\
D_{3}(u^{-5},u^2,q^6u^9;q)
&=D_{3}(u^{-5},u^3,q^6u^9;q)\label{equation:idN3-4}\\
&=\frac{u^5(q;q)_{\infty}(q^3;q^3)_{\infty}^4}{\Theta(u;q)\Theta(u^6;q^3)\Theta(qu^3;q^3)\Theta(q^6u^9;q^9)},
\notag \\
D_{3}(qz^{-4},z,z^9;q^2)
&=D_{3}(qz^{-4},qz^3,z^9;q^2)\label{equation:idN3-5}\\
&=-\frac{q^{-1}z^4(q^2;q^2)_{\infty}^2\Theta(q^3;q^{18})\Theta(-z;q)}{(q;q)_{\infty}\Theta(qz^3;q^2)\Theta(z^9;q^{18})},
\notag \\
D_{3}(x,q,q^3x^{-3};q^2)
&=D_{3}(x,q^{-1}x^{-1},q^3x^{-3};q^2)\label{equation:idN3-6}\\
&=-\frac{x^{-1}(q^6;q^6)_{\infty}^3\Theta(x;q^2)\Theta(q^4x^2;q^6)}
{(q^3;q^3)_{\infty}\Theta(x;q)\Theta(q^2x;q^6)\Theta(q^3x^{-3};q^{18})},
\notag \\
D_{3}(qz^{-2},z,z^3;q^2)
&=D_{3}(qz^{-2},qz,z^3;q^2)\label{equation:idN3-7}\\
&=-\frac{q^{-1}z(q^3;q^3)_{\infty}(q^6;q^6)_{\infty}^2\Theta(qz;q^2)\Theta(z^2;q^6)}
{\Theta(z;q)\Theta(q^3z;q^6)\Theta(q^3z^{3};q^{6})\Theta(z^{3};q^{18})},
\notag\\
D_{3}(x,q,q^9;q^2)
&=D_{3}(x,qx^{-1},q^9;q^2)\label{equation:idN3-8}\\
&=-\frac{qx^{-1}\Theta(-q;q^4)(q^3;q^3)_{\infty}(q^6;q^6)_{\infty}^2\Theta(x^2;q)}
{\Theta(q^9;q^{18})\Theta(x;q)\Theta(qx^{2};q^{2})\Theta(q^3x^{3};q^{6})}
\notag\\
D_{3}(qz^{-3},z,q^9z^6;q^2)
&=D_{3}(qz^{-3},qz^2,q^9z^6;q^2)\label{equation:idN3-9}\\
&=-\frac{z\Theta(-q;q^4)(q^3;q^3)_{\infty}(q^6;q^6)_{\infty}^2\Theta(z^2;q^3)}
{\Theta(qz^2;q^2)\Theta(z;q^{3})\Theta(z^{3};q^{6})\Theta(q^9z^6;q^{18})},
\notag\\
D_{3}(qu^{-5},u^3,u^9;q^2)
&=D_{3}(qu^{-5},qu^2,u^9;q^2)
\label{equation:idN3-10}\\
&=-\frac{q^{-1}u^4(q;q)_{\infty}(q^2;q^2)_{\infty}(q^6;q^6)_{\infty}^3\Theta(u^2;q)\Theta(u^3;q^6)_{\infty}}
{(q^3;q^3)_{\infty}\Theta(u;q)\Theta(qu^2;q^{2})\Theta(q^3u^{6};q^{6})\Theta(u^3;q^{2})\Theta(u^9;q^{18})},\notag \\
D_{3}(qz^{-4},z,q^9z^9;q^2)
&=D_{3}(qz^{-4},qz^3,q^9z^9;q^2)
\label{equation:idN3-11}\\
&=-\frac{z(q;q)_{\infty}(q^2;q^2)_{\infty}(q^6;q^6)_{\infty}^3\Theta(z^2;q)\Theta(q^3z^3;q^6)}
{(q^3;q^3)_{\infty}\Theta(z;q)\Theta(qz^2;q^{2})\Theta(z^{3};q^{6})\Theta(qz^3;q^{2})\Theta(q^9z^9;q^{18})}.\notag
\end{align}
\end{theorem}

\section{Statement of results: the case $n=4$}\label{section:N4}

We point out the $n=4$ specialization of \cite[Theorem $3.5$]{HM} and then present a family of specializations that evaluates to a single quotient of theta functions.

\begin{corollary} \label{corollary:msplitn4zprime} For generic $x,z,z'\in \mathbb{C}^*$ 
{\allowdisplaybreaks \begin{align}
D_4(x,z,z';q)&=\frac{z' (q^4;q^4)_{\infty}^3}{\Theta(xz;q) \Theta(z';q^{16}) \Theta\big(-q^{6} x^4 z';q^4)} \Big [ 
% r=0
\frac{\Theta\big(-q^{6} x^4 z z';q^4\big)
\Theta( z^4/z';q^{16})}
{\Theta( z;q^4\big )}\label{equation:msplit4}\\
% r=1
&\qquad -xz\frac{
\Theta\big(-q^{7} x^4 z z';q^4\big)
\Theta(q^{4} z^4/z';q^{16})}
{\Theta(q z;q^4\big )}\notag \\
% r=2
&\qquad +q x^2z^2\frac{
\Theta\big(-q^{8} x^4 z z';q^4\big)
\Theta(q^{8} z^4/z';q^{16})}
{\Theta(q^2 z;q^4\big )}\notag \\
% r=3
&\qquad -q^{3}x^3z^3\frac{
\Theta\big(-q^{9} x^4 z z';q^4\big)
\Theta(q^{12} z^4/z';q^{16})}
{\Theta(q^3 z;q^4\big )}\Big ],\notag
\end{align}}%
where
\begin{align}
D_4(x,z,z';q)&:=m(x,z;q) -   m\big (-q^{6} x^4, z';q^{16} \big) 
 +  q^{-1} x m\big(-q^{2} x^4, z';q^{16} \big) \\
&\qquad -  q^{-3} x^2 m\big(-q^{-2} x^4, z';q^{16} \big) 
 +  q^{-6} x^3 m\big(-q^{-6} x^4, z';q^{16} \big).\notag
\end{align}
\end{corollary}
We have the following specialization that yields a single quotients of theta functions.
\begin{theorem}\label{theorem:caseN4}  We have
\begin{align}
D_{4}(-z^{-5},z,z^{16};q)
&=D_{4}(-z^{-5},-z^4,z^{16};q)
\label{equation:idN4-1}\\
&=\frac{q^{-1}z^{10}(q;q)_{\infty}(q^4;q^4)_{\infty}^3(q^{16};q^{16})_{\infty}\Theta(z^2;q)}
{(q^{8};q^{8})_{\infty}\Theta(z;q)\Theta(-qz^{4};q^2)\Theta(-q^2z^{4};q^{4})\Theta(z^{16};q^{16})}.\notag
\end{align}
\end{theorem}

\section{Preliminaries}\label{section:prelim}
We will frequently use the following identities without mention.  They easily follow from the definitions.
{\allowdisplaybreaks \begin{subequations}
\begin{gather}
\overline{\Theta}_{0,1}=2\overline{\Theta}_{1,4}=\frac{2\Theta_2^2}{\Theta_1},  \ 
\overline{\Theta}_{1,2}=\frac{\Theta_2^5}{\Theta_1^2\Theta_4^2}, \ 
  \Theta_{1,2}=\frac{\Theta_1^2}{\Theta_2},   \overline{\Theta}_{1,3}=\frac{\Theta_2\Theta_3^2}{\Theta_1\Theta_6},
   \notag\\
\Theta_{1,4}=\frac{\Theta_1\Theta_4}{\Theta_2},  
\  \Theta_{1,6}=\frac{\Theta_1\Theta_6^2}{\Theta_2\Theta_3},   \ 
\overline{\Theta}_{1,6}=\frac{\Theta_2^2\Theta_3\Theta_{12}}{\Theta_1\Theta_4\Theta_6}.\notag
\end{gather}
\end{subequations}}%
Also following from the definitions are the following general identities:
{\allowdisplaybreaks \begin{subequations}
\begin{gather}
\Theta(q^n x;q)=(-1)^nq^{-\binom{n}{2}}x^{-n}j(x;q), \ \ n\in\mathbb{Z},\label{equation:j-elliptic}\\
\Theta(x;q)=\Theta(q/x;q)\label{equation:j-flip},\\
\Theta(x;q)={\Theta_1}\Theta(x,qx,\dots,q^{n-1}x;q^n)/{\Theta_n^n} \ \ {\text{if $n\ge 1$,}}\label{equation:j-mod}\\
\Theta(z;q)=\sum_{k=0}^{m-1}(-1)^k q^{\binom{k}{2}}z^k
\Theta\big ((-1)^{m+1}q^{\binom{m}{2}+mk}z^m;q^{m^2}\big ),\label{equation:j-split}\\
\Theta(x^n;q^n)={\Theta_n}\Theta(x,\zeta_nx,\dots,\zeta_n^{n-1}x;q^n)/{\Theta_1^n} \ \ {\text{if $n\ge 1$.}}\label{equation:j-roots}
\end{gather}
\end{subequations}}

We collect several well-known results about theta functions in terms of a proposition.
\begin{proposition}   For generic $x,y\in \mathbb{C}^*$ 
 \begin{subequations}
{\allowdisplaybreaks \begin{gather}
\Theta(qx^3;q^3)+x\Theta(q^2x^3;q^3)=\Theta(-x;q)\Theta(qx^2;q^2)/\Theta_2={\Theta_1\Theta(x^2;q)}/{\Theta(x;q)},\label{equation:H1Thm1.0}\\
\Theta(x;q)\Theta(y;q)=\Theta(-xy;q^2)\Theta(-qx^{-1}y;q^2)-x\Theta(-qxy;q^2)\Theta(-x^{-1}y;q^2),\label{equation:H1Thm1.1}\\
\Theta(-x;q)\Theta(y;q)-\Theta(x;q)\Theta(-y;q)=2x\Theta(x^{-1}y;q^2)j(qxy;q^2),\label{equation:H1Thm1.2A}\\
\Theta(-x;q)\Theta(y;q)+\Theta(x;q)\Theta(-y;q)=2\Theta(xy;q^2)\Theta(qx^{-1}y;q^2).\label{equation:H1Thm1.2B}
\end{gather}}%
\end{subequations}
\end{proposition}

\noindent Identity (\ref{equation:H1Thm1.0}) is the quintuple product identity.

We recall the three-term Weierstrass relation for theta functions \cite[(1.)]{We}, \cite{Ko}.  First we define some more shorthand
\begin{equation*}
\Theta(x_1,x_2,\dots,x_n;q):=\Theta(x_1;q)\Theta(x_2;q)\cdots\Theta(x_n;q).
\end{equation*}
The three-term Weierstrass relation for theta functions then reads
\begin{proposition}\label{proposition:Weierstrass-id} For generic $a,b,c,d\in \mathbb{C}^*$
\begin{equation}
\Theta(ac,a/c,bd,b/d;q)=\Theta(ad,a/d,bc,b/c;q)+b/c \cdot \Theta(ab,a/b,cd,c/d;q).\label{equation:Weierstrass}
\end{equation}
\end{proposition}

\begin{corollary}  We have
\begin{align}
\Theta(q^2,z^2,qz,q/z;q^6)&=\Theta(qz^2,q,q^2/z,z;q^6)+z \cdot \Theta(q^2z,z,q,q/z^2;q^6),\label{equation:WR-cor1}\\
\Theta(q^5/z,q^3z,qz^2,q;q^6)&=\Theta(q^4z,q^4/z,q^2,z^2;q^6)+z^2 \cdot \Theta(q^5z,q^3/z,q,q/z^2;q^6).
\label{equation:WR-cor2}
\end{align}
\end{corollary}
\begin{proof}  For (\ref{equation:WR-cor1}), we set $q\to q^6$ and $(a,b,c,d)\to (qz,q,q/z,z)$.  For (\ref{equation:WR-cor2}), we set $q\to q^6$ and $(a,b,c,d)\to (q^4,qz,q/z,z)$. \qedhere

\end{proof}

The next proposition follows immediately from \cite[Lemma $2$]{ASD} see also \cite[Theorem $1.7$]{H1}.
\begin{proposition}\label{proposition:H1Thm1.7} Let $C$ be a nonzero complex number, and let $n$ be a nonnegative integer.  Suppose that $F(z)$ is analytic for $z\ne 0$ and satisfies $F(qz)=Cz^{-n}F(z)$.  Then either $F(z)$ has exactly $n$ zeros in the annulus $|q|<|z|\le 1$ or $F(z)=0$ for all $z$.
\end{proposition}

\begin{corollary}  We have
\begin{equation}
 \Theta(-q^2;q^6)\Theta(q^6z;q^{12})\Theta(z;q^6)+q\Theta(-q;q^6)\Theta(z;q^{12})\Theta(q^3z;q^6)
 -\frac{\Theta(z;q^4)(q^3;q^3)_{\infty}^3}{(q;q)_{\infty}}=0.\label{equation:ASD-cor1}
\end{equation}
\end{corollary}

\begin{proof}
We let the left-hand side be $f(z)$.  We see that $f(q^{12}z)=-q^{-12}z^{-3}f(z)$.   By Proposition \ref{proposition:H1Thm1.7}, $f(z)$ has either exactly $3$ zeros in the annulus $|q|^{12}<|z|\le 1$ or $f(z)=0$ for all $z$.  But there are at least four such values of $z$ for which at least one of the terms in $f(z)$ vanishes: $0, q^3,q^{6},q^{9}$.  Verifying that $f(z)$ vanishes for these four values of $z$ is just a matter of proving that the remaining two theta products sum to zero.
\end{proof}

The function $D_n$ satisfies several easily shown functional equations
\begin{proposition} We have
\begin{gather}
D_n(x,z,z';q)=D_n(x,qz,z';q)=D_n(x,x^{-1}z^{-1},z';q)=D_n(x,z,q^{n^2}z';q),\label{equation:Dn-funcEqn1}\\
\Theta(x;q)D_n(x,z,z';q)=\Theta(qx;q)D_n(qx,z,z';q).\label{equation:Dn-funcEqn2}
\end{gather}
\end{proposition}
\section{Proof of Theorem \ref{theorem:caseN2}}\label{section:N2-proofs}

\subsection{Proof of (\ref{equation:idN2-1})}
We specialize (\ref{equation:msplit2}).  This yields
\begin{align*}
D_2(x,z,z^2;q)
&=\frac{z^2(q^2;q^2)_{\infty}^3}{\Theta(xz;q)\Theta(z^2;q^4)}\Big [
\frac{\Theta(-qx^2z^3;q^2)\Theta(1;q^{4})}{\Theta(-qx^2z^2;q^2)\Theta(z;q^2)}
-xz \frac{\Theta(-q^2x^2z^3;q^2)\Theta(q^2;q^{4})}{\Theta(-qx^2z^2;q^2)\Theta(qz;q^2)}\Big ]\\
&=-xz\frac{z^2(q^2;q^2)_{\infty}^3\Theta(-q^2x^2z^3;q^2)\Theta(q^2;q^{4})}
{\Theta(xz;q)\Theta(z^2;q^4)\Theta(-qx^2z^2;q^2)\Theta(qz;q^2)}.
\end{align*}

\subsection{Proof of (\ref{equation:idN2-2})}
This is just \cite[Corollary $3.7$]{HM}.

\subsection{Proof of (\ref{equation:idN2-3})}
We specialize (\ref{equation:msplit2}).  This yields
{\allowdisplaybreaks \begin{align*}
D_2&(x,z,x^{-1};q)\\
&=\frac{x^{-1}(q^2;q^2)_{\infty}^3}{\Theta(xz;q)\Theta(x^{-1};q^4)}\Big [
\frac{\Theta(-qxz;q^2)\Theta(xz^2;q^{4})}{\Theta(-qx;q^2)\Theta(z;q^2)}
 -xz \frac{\Theta(-q^2xz;q^2)\Theta(q^2xz^2;q^{4})}{\Theta(-qx;q^2)\Theta(qz;q^2)}\Big ]\\
&=-\frac{(q^2;q^2)_{\infty}^3}{\Theta(xz;q)\Theta(x;q^4)}\Big [
\frac{\Theta(-qxz;q^2)\Theta(xz^2;q^{4})}{\Theta(-qx;q^2)\Theta(z;q^2)}
-xz \frac{\Theta(-q^2xz;q^2)\Theta(q^2xz^2;q^{4})}{\Theta(-qx;q^2)\Theta(qz;q^2)}\Big ],
\end{align*}}%
where we have used (\ref{equation:j-flip}) and (\ref{equation:j-elliptic}).  Combining fractions yields
\begin{align*}
D_2&(x,z,x^{-1};q)\\
&=-\frac{(q^2;q^2)_{\infty}^3}
{\Theta(xz;q)\Theta(x;q^4)\Theta(-qx;q^2)\Theta(z;q^2)\Theta(qz;q^2)}\\
&\qquad \cdot \Big [
\Theta(-qxz;q^2)\Theta(xz^2;q^{4})\Theta(qz;q^2)
 -xz\Theta(-q^2xz;q^2)\Theta(q^2xz^2;q^{4})\Theta(z;q^2)\Big ]\\
&=-\frac{(q;q)_{\infty}(q^2;q^2)_{\infty}}
{\Theta(xz;q)\Theta(x;q^4)\Theta(-qx;q^2)\Theta(z;q)}\\
&\qquad \cdot \Big [
\Theta(-qxz;q^2)\Theta(xz^2;q^{4})\Theta(qz;q^2) 
 -xz\Theta(-q^2xz;q^2)\Theta(q^2xz^2;q^{4})\Theta(z;q^2)\Big ],
\end{align*}
where we have used (\ref{equation:j-mod}).  Using (\ref{equation:Weierstrass}) with $q\to q^2$, $a\to iqz\sqrt{x}$, $b\to iz\sqrt{x}$, $c\to i\sqrt{x}$, $d\to i$, gives
\begin{align*}
D_2(x,z,x^{-1};q)
&=-\frac{(q;q)_{\infty}(q^2;q^2)_{\infty}}
{\Theta(xz;q)\Theta(x;q^4)\Theta(-qx;q^2)\Theta(z;q)}\cdot \frac{(q^4;q^4)_{\infty}}{(q^2;q^2)_{\infty}^2}\\
&\qquad \cdot z\cdot \Theta(-qxz^2;q^2)\Theta(q;q^2)\Theta(-\sqrt{x};q^2)\Theta(\sqrt{x};q^2)\\
&=-\frac{z(q;q)_{\infty}^3\Theta(-qxz^2;q^2)}
{\Theta(xz;q)\Theta(-qx;q^2)\Theta(z;q)},
\end{align*}
where we have used (\ref{equation:j-mod}) and elementary product rearrangements.

\subsection{Proof of (\ref{equation:idN2-4})}
We specialize (\ref{equation:msplit2}).  This yields
{\allowdisplaybreaks \begin{align*}
D_2&(x,z,x^{-2};q)\\
&=\frac{x^{-2}(q^2;q^2)_{\infty}^3}{\Theta(xz;q)\Theta(x^{-2};q^4)}\Big [
\frac{\Theta(-qz;q^2)\Theta(x^2z^2;q^{4})}{\Theta(-q;q^2)\Theta(z;q^2)}
 -xz \frac{\Theta(-q^2z;q^2)\Theta(q^2x^2z^2;q^{4})}{\Theta(-q;q^2)\Theta(qz;q^2)}\Big ]\\
&=-\frac{(q^2;q^2)_{\infty}^3}{\Theta(xz;q)\Theta(x^{2};q^4)}\Big [
\frac{\Theta(-qz;q^2)\Theta(x^2z^2;q^{4})}{\Theta(-q;q^2)\Theta(z;q^2)}
 -xz \frac{\Theta(-q^2z;q^2)\Theta(q^2x^2z^2;q^{4})}{\Theta(-q;q^2)\Theta(qz;q^2)}\Big ],
\end{align*}}%
where we have used (\ref{equation:j-flip}) and (\ref{equation:j-elliptic}).  Using elementary product rearrangements and combining fractions yields
{\allowdisplaybreaks \begin{align*}
D_2&(x,z,x^{-2};q)\\
&=-\frac{(q^2;q^2)_{\infty}^3}{\Theta(xz;q)\Theta(x^{2};q^4)\Theta(z;q^2)\Theta(qz;q^2)}\cdot 
\frac{(q;q)_{\infty}^2(q^4;q^4)_{\infty}^2}{(q^2;q^2)_{\infty}^5}\\
&\qquad \cdot \Big [
\Theta(-qz;q^2)\Theta(x^2z^2;q^{4})\Theta(qz;q^2)
 -xz \cdot \Theta(-q^2z;q^2)\Theta(q^2x^2z^2;q^{4})\Theta(z;q^2)\Big ]\\
&=-\frac{1}{\Theta(xz;q)\Theta(x^{2};q^4)\Theta(z;q)}\cdot 
\frac{(q;q)_{\infty}^3(q^4;q^4)_{\infty}^2}{(q^2;q^2)_{\infty}^4}\\
&\qquad \cdot \Big [
\Theta(-qz;q^2)\Theta(x^2z^2;q^{4})\Theta(qz;q^2)
 -xz \cdot \Theta(-q^2z;q^2)\Theta(q^2x^2z^2;q^{4})\Theta(z;q^2)\Big ],
\end{align*}}%
where we have used (\ref{equation:j-mod}).  Using (\ref{equation:j-elliptic}) and (\ref{equation:j-roots}) yields 
\begin{align*}
D_2(x,z,x^{-2};q)
&=-\frac{1}{\Theta(xz;q)\Theta(x^{2};q^4)\Theta(z;q)}\cdot 
\frac{(q;q)_{\infty}^3(q^4;q^4)_{\infty}}{(q^2;q^2)_{\infty}^2}\\
&\qquad \cdot \Big [
\Theta(q^2z^2;q^4)\Theta(x^2z^2;q^{4})
-x \cdot \Theta(z^2;q^4)\Theta(q^2x^2z^2;q^{4})\Big ].
\end{align*}
Using (\ref{equation:H1Thm1.1}) with $q\to q^2$, $x\to x$, $y=\to -xz^2$ gives
\begin{align*}
D_2(x,z,x^{-2};q)
&=-\frac{1}{\Theta(xz;q)\Theta(x^{2};q^4)\Theta(z;q)}\cdot 
\frac{(q;q)_{\infty}^3(q^4;q^4)_{\infty}}{(q^2;q^2)_{\infty}^2}
\cdot \Theta(x;q^2)\Theta(-xz^2;q^2)\\
&=-\frac{(q;q)_{\infty}^3\Theta(-xz^2;q^2)}{\Theta(xz;q)\Theta(-x;q^2)\Theta(z;q)},
\end{align*}
where we have used (\ref{equation:j-roots}).

\subsection{Proof of (\ref{equation:idN2-5})}
We specialize (\ref{equation:msplit2}).  This yields
\begin{align*}
D_2&(u^3,z,u^{-4};q)\\
&=\frac{u^{-4}(q^2;q^2)_{\infty}^3}{\Theta(u^3z;q)\Theta(u^{-4};q^4)}\Big [
\frac{\Theta(-qu^2z;q^2)\Theta(u^4z^2;q^{4})}{\Theta(-qu^2;q^2)\Theta(z;q^2)}
 -u^3z \frac{\Theta(-q^2u^2z;q^2)\Theta(q^2u^4z^2;q^{4})}{\Theta(-qu^2;q^2)\Theta(qz;q^2)}\Big ]\\
&=-\frac{(q^2;q^2)_{\infty}^3}{\Theta(u^3z;q)\Theta(u^{4};q^4)}\Big [
\frac{\Theta(-qu^2z;q^2)\Theta(u^4z^2;q^{4})}{\Theta(-qu^2;q^2)\Theta(z;q^2)}
 -u^3z \frac{\Theta(-q^2u^2z;q^2)\Theta(q^2u^4z^2;q^{4})}{\Theta(-qu^2;q^2)\Theta(qz;q^2)}\Big ],
\end{align*}
where we have used (\ref{equation:j-elliptic}).  Combining fractions yields
{\allowdisplaybreaks \begin{align*}
D_2&(u^3,z,u^{-4};q)\\
&=-\frac{(q^2;q^2)_{\infty}^3}
{\Theta(u^3z;q)\Theta(u^{4};q^4)\Theta(-qu^2;q^2)\Theta(z;q^2)\Theta(qz;q^2)}\\
&\qquad \cdot \Big [
\Theta(-qu^2z;q^2)\Theta(u^4z^2;q^{4})\Theta(qz;q^2)
 -u^3z \Theta(-q^2u^2z;q^2)\Theta(q^2u^4z^2;q^{4})\Theta(z;q^2)\Big ]\\
&=-\frac{(q^2;q^2)_{\infty}^3}
{\Theta(u^3z;q)\Theta(u^{4};q^4)\Theta(-qu^2;q^2)\Theta(z;q)}
\cdot \frac{(q;q)_{\infty}}{(q^2;q^2)_{\infty}^2}\\
&\qquad \cdot \Big [
\Theta(-qu^2z;q^2)\Theta(u^4z^2;q^{4})\Theta(qz;q^2)
 -u \Theta(-u^2z;q^2)\Theta(q^2u^4z^2;q^{4})\Theta(z;q^2)\Big ],
\end{align*}}%
where we have used (\ref{equation:j-mod}) and (\ref{equation:j-elliptic}).  Using (\ref{equation:j-roots}) yields
{\allowdisplaybreaks \begin{align*}
D_2(u^3,z,u^{-4};q)
&=-\frac{(q^2;q^2)_{\infty}^3}
{\Theta(u^3z;q)\Theta(u^{4};q^4)\Theta(-qu^2;q^2)\Theta(z;q)}
\cdot \frac{(q;q)_{\infty}}{(q^2;q^2)_{\infty}^2}\cdot \frac{(q^4;q^4)_{\infty}}{(q^2;q^2)_{\infty}^2}\\
&\qquad \cdot \Big [
\Theta(-qu^2z;q^2)\Theta(u^2z;q^{2})\Theta(-u^2z;q^{2})\Theta(qz;q^2)\\
&\qquad \qquad -u \Theta(-u^2z;q^2)\Theta(qu^2z;q^{2})\Theta(-qu^2z;q^{2})\Theta(z;q^2)\Big ]\\
&=-\frac{(q;q)_{\infty}(q^4;q^4)_{\infty}\Theta(-qu^2z;q^2)\Theta(-u^2z;q^{2})}
{\Theta(u^3z;q)\Theta(u^{4};q^4)\Theta(-qu^2;q^2)\Theta(z;q)(q^2;q^2)_{\infty}}
\cdot\\
&\qquad \cdot \Big [
\Theta(u^2z;q^{2})\Theta(qz;q^2) -u \Theta(qu^2z;q^{2})\Theta(z;q^2)\Big ].
\end{align*}}%
Using (\ref{equation:j-mod}) and then (\ref{equation:H1Thm1.1}) with $x\to u$, $y\to -zu$ yields
\begin{align*}
D_2(u^3,z,u^{-4};q)
&=-\frac{(q;q)_{\infty}(q^4;q^4)_{\infty}\Theta(-u^2z;q)}
{\Theta(u^3z;q)\Theta(u^{4};q^4)\Theta(-qu^2;q^2)\Theta(z;q)(q^2;q^2)_{\infty}}
\cdot \frac{(q^2;q^2)_{\infty}^2}{(q;q)_{\infty}}\\
\cdot\\
&\qquad \cdot \Big [
\Theta(u^2z;q^{2})\Theta(qz;q^2) -u \Theta(qu^2z;q^{2})\Theta(z;q^2)\Big ]\\
&=-\frac{(q^2;q^2)_{\infty}(q^4;q^4)_{\infty}\Theta(-u^2z;q)}
{\Theta(u^3z;q)\Theta(u^{4};q^4)\Theta(-qu^2;q^2)\Theta(z;q)}
\cdot \Theta(u;q)\Theta(-zu;q).
\end{align*}

\subsection{Proof of (\ref{equation:idN2-6})}
We specialize (\ref{equation:msplit2}).  This yields
\begin{align*}
D_2&(qz^{-2},z,z^3;q^3)\\
&=\frac{z^3(q^6;q^6)_{\infty}^3}{\Theta(qz^{-1};q^3)\Theta(z^3;q^{12})}
\Big [\frac{\Theta(-q^5;q^6)\Theta(z^{-1};q^{12})}{\Theta(-q^5z^{-1};q^6)\Theta(z;q^6)}
 -qz^{-1} \frac{\Theta(-q^8;q^6)\Theta(q^6z^{-1};q^{12})}{\Theta(-q^5z^{-1};q^6)\Theta(q^3z;q^6)}\Big ]\\
&=\frac{z^3(q^6;q^6)_{\infty}^3}{\Theta(q^2z;q^3)\Theta(z^3;q^{12})}
\Big [-z^{-1}\frac{\Theta(-q^5;q^6)\Theta(z;q^{12})}{\Theta(-qz;q^6)\Theta(z;q^6)}
 -q^{-1}z^{-1} \frac{\Theta(-q^2;q^6)\Theta(q^6z;q^{12})}{\Theta(-qz;q^6)\Theta(q^3z;q^6)}\Big ],
\end{align*}
where we have used (\ref{equation:j-flip}) and (\ref{equation:j-elliptic}).   Combining fractions and employing (\ref{equation:j-mod}) yields
\begin{align*}
D_2(qz^{-2},z,z^3;q^3)
&=-\frac{q^{-1}z^2(q^6;q^6)_{\infty}^3}
{\Theta(q^2z;q^3)\Theta(z^3;q^{12})\Theta(-qz;q^6)\Theta(z;q^6)\Theta(q^3z;q^6)} \\
&\qquad \cdot \Big [q\Theta(-q^5;q^6)\Theta(z;q^{12})\Theta(q^3z;q^6)
+\Theta(-q^2;q^6)\Theta(q^6z;q^{12})\Theta(z;q^6)\Big ]\\
&=-\frac{q^{-1}z^2(q^6;q^6)_{\infty}^3}
{\Theta(q^2z;q^3)\Theta(z^3;q^{12})\Theta(-qz;q^6)\Theta(z;q^3)} 
\cdot \frac{(q^3;q^3)_{\infty}}{(q^6;q^6)_{\infty}^2}\\
&\qquad \cdot \Big [q\Theta(-q^5;q^6)\Theta(z;q^{12})\Theta(q^3z;q^6)
+\Theta(-q^2;q^6)\Theta(q^6z;q^{12})\Theta(z;q^6)\Big ].
\end{align*}
Using (\ref{equation:ASD-cor1}) yields
\begin{align*}
D_2(qz^{-2},z,z^3;q^3)
&=-\frac{q^{-1}z^2(q^3;q^3)_{\infty}^4(q^6;q^6)_{\infty}\Theta(z;q^4)}
{(q;q)_{\infty}\Theta(q^2z;q^3)\Theta(z^3;q^{12})\Theta(-qz;q^6)\Theta(z;q^3)}.
\end{align*}

\subsection{Proof of (\ref{equation:idN2-7})}
We specialize (\ref{equation:msplit2}).  This yields
{\allowdisplaybreaks \begin{align*}
D_2&(q^2z^{-2},z,z^3;q^3)\\
&=\frac{z^3(q^6;q^6)_{\infty}^3}{\Theta(q^2z^{-1};q^3)\Theta(z^3;q^{12})}\Big [
\frac{\Theta(-q^7;q^6)\Theta(z^{-1};q^{12})}{\Theta(-q^7z^{-1};q^6)\Theta(z;q^6)}
 -q^2z^{-1} \frac{\Theta(-q^{10};q^6)\Theta(q^6z^{-1};q^{12})}{\Theta(-q^7z^{-1};q^6)\Theta(q^3z;q^6)}\Big ]\\
&=- \frac{q^{-2}z^2(q^6;q^6)_{\infty}^3}{\Theta(qz;q^3)\Theta(z^3;q^{12})}\Big [
q\frac{\Theta(-q;q^6)\Theta(z;q^{12})}{\Theta(-q^7z^{-1};q^6)\Theta(z;q^6)}
 +\frac{\Theta(-q^{4};q^6)\Theta(q^6z;q^{12})}{\Theta(-q^7z^{-1};q^6)\Theta(q^3z;q^6)}\Big ],
\end{align*}}%
where we have used (\ref{equation:j-flip}), (\ref{equation:j-elliptic}), and simplified.  Combining fractions and using (\ref{equation:j-mod}) yields
\begin{align*}
D_2(q^2z^{-2},z,z^3;q^3)
&=- \frac{q^{-2}z^2(q^6;q^6)_{\infty}^3}
{\Theta(qz;q^3)\Theta(z^3;q^{12})\Theta(-q^7z^{-1};q^6)\Theta(z;q^6)\Theta(q^3z;q^6)}
\\
&\qquad \cdot \Big [
q\Theta(-q;q^6)\Theta(z;q^{12})\Theta(q^3z;q^6)
+\Theta(-q^{4};q^6)\Theta(q^6z;q^{12})\Theta(z;q^6)\Big ]\\
&=- \frac{q^{-2}z^2(q^6;q^6)_{\infty}^3}
{\Theta(qz;q^3)\Theta(z^3;q^{12})\Theta(-q^7z^{-1};q^6)\Theta(z;q^3)}
\cdot \frac{(q^3;q^3)_{\infty}}{(q^6;q^6)_{\infty}^2}\\
&\qquad \cdot \Big [
q\Theta(-q;q^6)\Theta(z;q^{12})\Theta(q^3z;q^6)
+\Theta(-q^{4};q^6)\Theta(q^6z;q^{12})\Theta(z;q^6)\Big ].
\end{align*}
Using (\ref{equation:ASD-cor1}) and then (\ref{equation:j-elliptic}) and (\ref{equation:j-flip}) yields
\begin{align*}
D_2(q^2z^{-2},z,z^3;q^3)
&=- \frac{q^{-2}z^2(q^3;q^3)_{\infty}(q^6;q^6)_{\infty}}
{\Theta(qz;q^3)\Theta(z^3;q^{12})\Theta(-q^7z^{-1};q^6)\Theta(z;q^3)}
 \cdot \frac{\Theta(z;q^4)(q^3;q^3)_{\infty}^3}{(q;q)_{\infty}}\\
 &=- \frac{q^{-1}z(q^3;q^3)_{\infty}^4(q^6;q^6)_{\infty}\Theta(z;q^4)}
{(q;q)_{\infty}\Theta(qz;q^3)\Theta(z^3;q^{12})\Theta(-q^5z;q^6)\Theta(z;q^3)}.
\end{align*}

\subsection{Proof of (\ref{equation:idN2-8})}
We specialize (\ref{equation:msplit2}).  This yields
\begin{align*}
D_2(x,z,-qx^{-2}z^{-2};q)
&=-\frac{qx^{-2}z^{-2}(q^2;q^2)_{\infty}^3}{\Theta(xz;q)\Theta(-qx^{-2}z^{-2};q^4)}\Big [
\frac{\Theta(q^2z^{-1};q^2)\Theta(-q^{-1}x^2z^4;q^{4})}{\Theta(q^2z^{-2};q^2)\Theta(z;q^2)}\\
&\qquad  -xz \frac{\Theta(q^3z^{-1};q^2)\Theta(-qx^2z^4;q^{4})}{\Theta(q^2z^{-2};q^2)\Theta(qz;q^2)}\Big ]\\
&=-\frac{qx^{-2}z^{-2}(q^2;q^2)_{\infty}^3}{\Theta(xz;q)\Theta(-q^3x^{2}z^{2};q^4)}\Big [
\frac{\Theta(z;q^2)\Theta(-q^{-1}x^2z^4;q^{4})}{\Theta(z^{2};q^2)\Theta(z;q^2)}\\
&\qquad +q^{-1}xz^2 \frac{\Theta(qz;q^2)\Theta(-qx^2z^4;q^{4})}{\Theta(z^{2};q^2)\Theta(qz;q^2)}\Big ],
\end{align*}
where we have used (\ref{equation:j-elliptic}) and (\ref{equation:j-flip}).  Simplifying yields
\begin{align*}
D_2(x,z,-qx^{-2}z^{-2};q)
&=-\frac{qx^{-2}z^{-2}(q^2;q^2)_{\infty}^3}{\Theta(xz;q)\Theta(-q^3x^{2}z^{2};q^4)\Theta(z^{2};q^2)}\\
&\qquad \cdot \Big [
\Theta(-q^{-1}x^2z^4;q^{4})
 +q^{-1}xz^2\Theta(-qx^2z^4;q^{4})\Big ].
\end{align*}
Using (\ref{equation:j-elliptic}) yields 
{\allowdisplaybreaks \begin{align*}
D_2(x,z,-qx^{-2}z^{-2};q)
&=-\frac{qx^{-2}z^{-2}(q^2;q^2)_{\infty}^3}{\Theta(xz;q)\Theta(-q^3x^{2}z^{2};q^4)\Theta(z^{2};q^2)}\\
&\qquad \cdot \Big [
\Theta(-q^3q^{-4}x^2z^4;q^{4})
 +q^{-1}xz^2\Theta(-qx^2z^4;q^{4})\Big ]\\
 &=-\frac{qx^{-2}z^{-2}(q^2;q^2)_{\infty}^3}{\Theta(xz;q)\Theta(-q^3x^{2}z^{2};q^4)\Theta(z^{2};q^2)}\\
&\qquad \cdot \Big [
q^{-1}x^{2}z^4\Theta(-q^3x^2z^4;q^{4})
 +q^{-1}xz^2\Theta(-qx^2z^4;q^{4})\Big ].
 \end{align*}}%
 Simplifying and using (\ref{equation:j-split}) yields
{\allowdisplaybreaks \begin{align*}
D_2(x,z,-qx^{-2}z^{-2};q)
 &=-\frac{x^{-1}(q^2;q^2)_{\infty}^3}{\Theta(xz;q)\Theta(-q^3x^{2}z^{2};q^4)\Theta(z^{2};q^2)}\\
&\qquad \cdot \Big [
xz^2\Theta(-q^3x^2z^4;q^{4})
 +\Theta(-qx^2z^4;q^{4})\Big ]\\
  &=-\frac{x^{-1}(q^2;q^2)_{\infty}^3\Theta(-xz^2;q)}
  {\Theta(xz;q)\Theta(-q^3x^{2}z^{2};q^4)\Theta(z^{2};q^2)}.
\end{align*}}%

\section{Proof of Theorem \ref{theorem:caseN3}}\label{section:N3-proofs}

The first equality in each equation follows from (\ref{equation:Dn-funcEqn1}).
\subsection{Proof of (\ref{equation:idN3-1})}  We specialize (\ref{equation:msplit3}) using the second term in (\ref{equation:idN3-1}).  We have

\begin{align*}
D_{3}(z^{-4},z^3,q^3z^9;q)&=\frac{q^3z^9(q^3;q^3)_{\infty}^3}{\Theta(z^{-1};q)\Theta(q^3z^9;q^{9})\Theta(q^3z^{-3};q^3)}
\Big [ 
\frac{1}{z^3}\frac{\Theta(q^3;q^3)\Theta(q^{-3};q^{9})}{\Theta(z^3;q^3)}\\
&\ \ \ \ \ -\frac{z^{-4}}{q}\frac{\Theta(q^4;q^3)\Theta(1;q^{9})}{\Theta(qz^3;q^3)}
+\frac{z^{-5}}{q}\frac{\Theta(q^5;q^3)\Theta(q^{3};q^{9})}{\Theta(q^2z^3;q^3)}\Big ].
\end{align*}
We use the fact that $\Theta(q^3;q^3)=0$ and $\Theta(1;q^{9})=0$ to obtain
\begin{align*}
D_{3}(z^{-4},z^3,q^3z^9;q)&=\frac{q^3z^9(q^3;q^3)_{\infty}^3}{\Theta(z^{-1};q)\Theta(q^3z^9;q^{9})\Theta(q^3z^{-3};q^3)}
\Big [ \frac{z^{-5}}{q}\frac{\Theta(q^5;q^3)\Theta(q^{3};q^{9})}{\Theta(q^2z^3;q^3)}\Big ].
\end{align*}
We rewrite a few theta functions using (\ref{equation:j-flip}) yields
\begin{align*}
D_{3}(z^{-4},z^3,q^3z^9;q)&=\frac{q^3z^9(q^3;q^3)_{\infty}^3}{\Theta(qz;q)\Theta(q^3z^9;q^{9})\Theta(z^{3};q^3)}
\Big [ \frac{z^{-5}}{q}\frac{\Theta(q^5;q^3)\Theta(q^{3};q^{9})}{\Theta(q^2z^3;q^3)}\Big ].
\end{align*}
Using (\ref{equation:j-elliptic}) with $n=1$ and simplifying gives
\begin{align*}
D_{3}(z^{-4},z^3,q^3z^9;q)&=\frac{z^5(q^3;q^3)_{\infty}^3}{\Theta(z;q)\Theta(q^3z^9;q^{9})\Theta(z^{3};q^3)}
\cdot \frac{\Theta(q^2;q^3)\Theta(q^{3};q^{9})}{\Theta(q^2z^3;q^3)}.
\end{align*}
Noting that $\Theta(q^2;q^3)=\Theta(q;q^3)=(q;q)_{\infty}$, and $\Theta(q^3;q^9)=(q^3;q^3)_{\infty}$ brings us to
\begin{align*}
D_{3}(z^{-4},z^3,q^3z^9;q)&=\frac{z^5(q;q)_{\infty}(q^3;q^3)_{\infty}^4}
{\Theta(z;q)\Theta(q^3z^9;q^{9})\Theta(z^{3};q^3)\Theta(q^2z^3;q^3)}.\qedhere
\end{align*}

\subsection{Proof of (\ref{equation:idN3-2})}
We specialize (\ref{equation:msplit3}) using the second term in (\ref{equation:idN3-2}).  This yields
{\allowdisplaybreaks \begin{align*}
D_{3}(z^{-4},z^3,q^6z^9;q)&=\frac{q^6z^9(q^3;q^3)_{\infty}^3}{\Theta(z^{-1};q)\Theta(q^6z^9;q^{9})\Theta(q^6z^{-3};q^3)}
\Big [ 
\frac{1}{z^3}\frac{\Theta(q^6;q^3)\Theta(q^{-6};q^{9})}{\Theta(z^3;q^3)}\\
&\ \ \ \ \ -\frac{z^{-4}}{q}\frac{\Theta(q^7;q^3)\Theta(q^{-3};q^{9})}{\Theta(qz^3;q^3)}
+\frac{z^{-5}}{q}\frac{\Theta(q^8;q^3)\Theta(1;q^{9})}{\Theta(q^2z^3;q^3)}\Big ].
\end{align*}
We use that facts that $\Theta(q^6;q^3)=0$ and $\Theta(1;q^9)=0$ and then simplify to obtain
\begin{align*}
D_{3}(z^{-4},z^3,q^6z^9;q)&=-\frac{q^5z^5(q^3;q^3)_{\infty}^3}{\Theta(z^{-1};q)\Theta(q^6z^9;q^{9})\Theta(q^6z^{-3};q^3)}
 \cdot \frac{\Theta(q^7;q^3)\Theta(q^{-3};q^{9})}{\Theta(qz^3;q^3)}.
\end{align*}}%
Using (\ref{equation:j-flip}) produces
\begin{align*}
D_{3}(z^{-4},z^3,q^6z^9;q)&=-\frac{q^5z^5(q^3;q^3)_{\infty}^3}{\Theta(qz;q)\Theta(q^6z^9;q^{9})\Theta(q^{-3}z^{3};q^3)}
 \cdot \frac{\Theta(q^7;q^3)\Theta(q^{12};q^{9})}{\Theta(qz^3;q^3)}.
\end{align*}
We use (\ref{equation:j-elliptic}) to rewrite three of the theta functions.  This gives
\begin{align*}
D_{3}(z^{-4},z^3,q^6z^9;q)&=-\frac{q^5z^5(q^3;q^3)_{\infty}^3}
{(-1)z^{-1}\Theta(z;q)\Theta(q^6z^9;q^{9})(-1)q^{-3}z^{3}\Theta(z^{3};q^3)}\\
& \ \ \ \ \  \cdot \frac{(-1)^2q^{-3}q^{-2}\Theta(q^7;q^3)(-1)q^{-3}\Theta(q^{12};q^{9})}{\Theta(qz^3;q^3)}\\
&=\frac{z^3(q^3;q^3)_{\infty}^3}
{\Theta(z;q)\Theta(q^6z^9;q^{9})\Theta(z^{3};q^3)}
 \cdot \frac{\Theta(q^2;q^3)\Theta(q^{3};q^{9})}{\Theta(qz^3;q^3)}.
\end{align*}
Noting that $\Theta(q^2;q^3)=\Theta(q;q^3)=(q;q)_{\infty}$, $\Theta(q^3;q^9)=(q^3;q^3)_{\infty}$ brings us to
\begin{align*}
D_{3}(z^{-4},z^3,q^6z^9;q)
&=\frac{z^3(q;q)_{\infty}(q^3;q^3)_{\infty}^4}
{\Theta(z;q)\Theta(q^6z^9;q^{9})\Theta(z^{3};q^3)\Theta(qz^3;q^3)}.\qedhere
\end{align*}

\subsection{Proof of (\ref{equation:idN3-3})}
We specialize (\ref{equation:msplit3}) using the second term in (\ref{equation:idN3-3}).  This yields
\begin{align*}
D_{3}(u^{-5},u^3,q^3u^9;q)&=\frac{q^3u^9(q^3;q^3)_{\infty}^3}{\Theta(u^{-2};q)\Theta(q^3u^9;q^{9})\Theta(q^3u^{-6};q^3)}
\Big [ 
\frac{1}{u^3}\frac{\Theta(q^3u^{-3};q^3)\Theta(q^{-3};q^{9})}{\Theta(u^3;q^3)}\\
&\ \ \ \ \ -\frac{u^{-5}}{q}\frac{\Theta(q^4u^{-3};q^3)\Theta(1;q^{9})}{\Theta(qu^3;q^3)}
+\frac{u^{-7}}{q}\frac{\Theta(q^5u^{-3};q^3)\Theta(q^{3};q^{9})}{\Theta(q^2u^3;q^3)}\Big ].
\end{align*}
Noting that $\Theta(1;q^9)=0$ we get
\begin{align*}
D_{3}(u^{-5},u^3,q^3u^9;q)&=\frac{q^3u^9(q^3;q^3)_{\infty}^3}{\Theta(u^{-2};q)\Theta(q^3u^9;q^{9})\Theta(q^3u^{-6};q^3)}\\
& \ \ \ \ \ \cdot \Big [ 
\frac{1}{u^3}\frac{\Theta(q^3u^{-3};q^3)\Theta(q^{-3};q^{9})}{\Theta(u^3;q^3)}
+\frac{u^{-7}}{q}\frac{\Theta(q^5u^{-3};q^3)\Theta(q^{3};q^{9})}{\Theta(q^2u^3;q^3)}\Big ]\\
&=\frac{q^3u^9(q^3;q^3)_{\infty}^3}{\Theta(qu^{2};q)\Theta(q^3u^9;q^{9})\Theta(u^{6};q^3)}\\
& \ \ \ \ \ \cdot \Big [ 
\frac{1}{u^3}\frac{\Theta(u^{3};q^3)\Theta(q^{12};q^{9})}{\Theta(u^3;q^3)}
+\frac{u^{-7}}{q}\frac{\Theta(q^{5}u^{-3};q^3)\Theta(q^{3};q^{9})}{\Theta(q^2u^3;q^3)}\Big ],
\end{align*}
where for the last equality we used (\ref{equation:j-flip}).  Using (\ref{equation:j-elliptic}) with $n=1$ gives
{\allowdisplaybreaks \begin{align*}
D_{3}(u^{-5},u^3,q^3u^9;q)&=\frac{q^3u^9(q^3;q^3)_{\infty}^3}
{(-1)u^{-2}\Theta(u^{2};q)\Theta(q^3u^9;q^{9})\Theta(u^{6};q^3)}\\
& \qquad \cdot \Big [ 
\frac{1}{u^3}\frac{\Theta(u^{3};q^3)(-1)q^{-3}\Theta(q^{3};q^{9})}{\Theta(u^3;q^3)}\\
&\qquad \qquad+\frac{u^{-7}}{q}\frac{(-1)q^{-2}u^3\Theta(q^{2}u^{-3};q^3)\Theta(q^{3};q^{9})}{\Theta(q^2u^3;q^3)}\Big ]\\
&=\frac{q^3u^{11}(q^3;q^3)_{\infty}^3}{\Theta(u^{2};q)\Theta(q^3u^9;q^{9})\Theta(u^{6};q^3)}\\
& \qquad \cdot \Big [ 
q^{-3}u^{-3}\frac{\Theta(u^{3};q^3)\Theta(q^{3};q^{9})}{\Theta(u^3;q^3)}
+u^{-4}q^{-3}\frac{\Theta(q^{2}u^{-3};q^3)\Theta(q^{3};q^{9})}{\Theta(q^2u^3;q^3)}\Big ].
\end{align*}}%
We again use (\ref{equation:j-flip}) and then simplify to get
\begin{align*}
D_{3}(u^{-5},u^3,q^3u^9;q)
&=\frac{q^3u^{11}(q^3;q^3)_{\infty}^3}{\Theta(u^{2};q)\Theta(q^3u^9;q^{9})\Theta(u^{6};q^3)}\\
& \ \ \ \ \  \cdot \Big [ 
q^{-3}u^{-3}\frac{\Theta(u^{3};q^3)\Theta(q^{3};q^{9})}{\Theta(u^3;q^3)}
+u^{-4}q^{-3}\frac{\Theta(qu^{3};q^3)\Theta(q^{3};q^{9})}{\Theta(q^2u^3;q^3)}\Big ]\\
&=\frac{u^7(q^3;q^3)_{\infty}^3}{\Theta(u^{2};q)\Theta(q^3u^9;q^{9})\Theta(u^{6};q^3)}
\cdot \Big [ 
u\Theta(q^{3};q^{9})
+\frac{\Theta(qu^{3};q^3)\Theta(q^{3};q^{9})}{\Theta(q^2u^3;q^3)}\Big ].
\end{align*}
Noting that $\Theta(q^3;q^9)=(q^3;q^3)_{\infty}$ and then combining fractions produces
\begin{align*}
D_{3}(u^{-5},u^3,q^3u^9;q)
&=\frac{u^7(q^3;q^3)_{\infty}^4}{\Theta(u^{2};q)\Theta(q^3u^9;q^{9})\Theta(u^{6};q^3)}
 \cdot \frac{\Theta(qu^{3};q^3)+u\Theta(q^2u^3;q^3)}{\Theta(q^2u^3;q^3)}\\
&=\frac{u^7(q^3;q^3)_{\infty}^4}{\Theta(u^{2};q)\Theta(q^3u^9;q^{9})\Theta(u^{6};q^3)\Theta(q^2u^3;q^3)}
 \cdot \frac{(q;q)_{\infty}\Theta(u^2;q)}{\Theta(u;q)}\\
 &=\frac{u^7(q;q)_{\infty}(q^3;q^3)_{\infty}^4}{\Theta(q^3u^9;q^{9})\Theta(u^{6};q^3)\Theta(q^2u^3;q^3)\Theta(u;q)},
\end{align*}
where we used the quintuple product identity (\ref{equation:H1Thm1.0}) and then simplified.

\subsection{Proof of (\ref{equation:idN3-4})}
We specialize (\ref{equation:msplit3}) using the second term in (\ref{equation:idN3-4}).  This yields
{\allowdisplaybreaks \begin{align*}
D_{3}(u^{-5},u^3,q^6u^9;q)&=\frac{q^6u^9(q^3;q^3)_{\infty}^3}{\Theta(u^{-2};q)\Theta(q^6u^9;q^{9})\Theta(q^6u^{-6};q^3)}
\Big [ 
\frac{1}{u^3}\frac{\Theta(q^6u^{-3};q^3)\Theta(q^{-6};q^{9})}{\Theta(u^3;q^3)} \\
&\ \ \ \ \ -\frac{u^{-5}}{q}\frac{\Theta(q^7u^{-3};q^3)\Theta(q^{-3};q^{9})}{\Theta(qu^3;q^3)}
+\frac{u^{-7}}{q}\frac{\Theta(q^8u^{-3};q^3)\Theta(1;q^{9})}{\Theta(q^2u^3;q^3)}\Big ]\\
&=\frac{q^6u^9(q^3;q^3)_{\infty}^3}{\Theta(u^{-2};q)\Theta(q^6u^9;q^{9})\Theta(q^6u^{-6};q^3)}\\
& \ \ \ \ \ \cdot \Big [ 
u^{-3}\frac{\Theta(q^6u^{-3};q^3)\Theta(q^{-6};q^{9})}{\Theta(u^3;q^3)} 
 -u^{-5}q^{-1}\frac{\Theta(q^7u^{-3};q^3)\Theta(q^{-3};q^{9})}{\Theta(qu^3;q^3)}\Big ],
\end{align*}}%
where we have used the fact that $\Theta(1;q^9)=0$.  Using (\ref{equation:j-flip}) and then (\ref{equation:j-elliptic}) brings us to
{\allowdisplaybreaks \begin{align*}
D_{3}(u^{-5},u^3,q^6u^9;q)&=\frac{q^6u^9(q^3;q^3)_{\infty}^3}
{\Theta(qu^{2};q)\Theta(q^6u^9;q^{9})\Theta(q^{-3}u^{6};q^3)}\\
& \ \ \ \ \ \cdot \Big [ 
u^{-3}\frac{\Theta(q^{-3}u^{3};q^3)\Theta(q^{15};q^{9})}{\Theta(u^3;q^3)} 
 -u^{-5}q^{-1}\frac{\Theta(q^{-4}u^{3};q^3)\Theta(q^{12};q^{9})}{\Theta(qu^3;q^3)}\Big ]\\
&=\frac{q^6u^9(q^3;q^3)_{\infty}^3}
{\Theta(qu^{2};q)\Theta(q^6u^9;q^{9})\Theta(q^{-3}u^{6};q^3)}\\
& \ \ \ \ \ \cdot \Big [ 
-q^{-6}u^{-3}\frac{\Theta(q^{-3}u^{3};q^3)\Theta(q^{6};q^{9})}{\Theta(u^3;q^3)} 
 +u^{-5}q^{-4}\frac{\Theta(q^{-4}u^{3};q^3)\Theta(q^{3};q^{9})}{\Theta(qu^3;q^3)}\Big ].
\end{align*}}%
Noting that $\Theta(q^6;q^9)=\Theta(q^3;q^9)=(q^3;q^3)_{\infty}$ gives us
\begin{align*}
D_{3}(u^{-5},u^3,q^6u^9;q)&=\frac{q^6u^9(q^3;q^3)_{\infty}^4}
{\Theta(qu^{2};q)\Theta(q^6u^9;q^{9})\Theta(q^{-3}u^{6};q^3)}\\
& \ \ \ \ \ \cdot \Big [ 
-q^{-6}u^{-3}\frac{\Theta(q^{-3}u^{3};q^3)}{\Theta(u^3;q^3)} 
 +u^{-5}q^{-4}\frac{\Theta(q^{-4}u^{3};q^3)}{\Theta(qu^3;q^3)}\Big ].
\end{align*}
Using (\ref{equation:j-elliptic}) gives
{\allowdisplaybreaks \begin{align*}
D_{3}(u^{-5},u^3,q^6u^9;q)&=\frac{q^9u^5(q^3;q^3)_{\infty}^4}
{\Theta(u^{2};q)\Theta(q^6u^9;q^{9})\Theta(q^{-3}u^{6};q^3)}
 \cdot \Big [ 
q^{-9}\frac{\Theta(u^{3};q^3)}{\Theta(u^3;q^3)} 
 +q^{-9}u\frac{\Theta(q^{2}u^{3};q^3)}{\Theta(qu^3;q^3)}\Big ]\\
 &=\frac{u^5(q^3;q^3)_{\infty}^4}
{\Theta(u^{2};q)\Theta(q^6u^9;q^{9})\Theta(u^{6};q^3)}
 \cdot \Big [ 
1
 +u\frac{\Theta(q^{2}u^{3};q^3)}{\Theta(qu^3;q^3)}\Big ]\\
&=\frac{u^5(q^3;q^3)_{\infty}^4}
{\Theta(u^{2};q)\Theta(q^6u^9;q^{9})\Theta(u^{6};q^3)}
 \cdot \frac{\Theta(qu^3;q^3)+u\Theta(q^{2}u^{3};q^3)}{\Theta(qu^3;q^3)}.
\end{align*}}%
Using the quintuple product identity (\ref{equation:H1Thm1.0}) brings us to
{\allowdisplaybreaks \begin{align*}
D_{3}(u^{-5},u^3,q^6u^9;q)
 &=\frac{u^5(q^3;q^3)_{\infty}^4}
{\Theta(u^{2};q)\Theta(q^6u^9;q^{9})\Theta(u^{6};q^3)\Theta(qu^3;q^3)}
 \cdot \frac{(q;q)_{\infty}\Theta(u^2;q)}{\Theta(u;q)}\\
  &=\frac{u^5(q;q)_{\infty}(q^3;q^3)_{\infty}^4}
{\Theta(q^6u^9;q^{9})\Theta(u^{6};q^3)\Theta(qu^3;q^3)\Theta(u;q)}.
\end{align*}}%

\subsection{Proof of (\ref{equation:idN3-5})}
We specialize (\ref{equation:msplit3}) using the second term in (\ref{equation:idN3-5}).  This yields
{\allowdisplaybreaks \begin{align*}
D_{3}(qz^{-4},qz^3,z^9;q^2)&=\frac{z^9(q^6;q^6)_{\infty}^3}{\Theta(q^2z^{-1};q^2)\Theta(z^9;q^{18})\Theta(q^3z^{-3};q^6)}
\Big [ 
\frac{1}{qz^{3}}\frac{\Theta(q^4;q^6)\Theta(q^3;q^{18})}{\Theta(qz^3;q^6)} \\
&\ \ \ \ \ -\frac{qz^{-4}}{q^2}\frac{\Theta(q^6;q^6)\Theta(q^{9};q^{18})}{\Theta(q^3z^3;q^6)}
+\frac{q^3z^{-5}}{q^2}\frac{\Theta(q^8;q^6)\Theta(q^{15};q^{18})}{\Theta(q^5z^3;q^6)}\Big ]\\
&=\frac{z^9(q^6;q^6)_{\infty}^3}{\Theta(q^2z^{-1};q^2)\Theta(z^9;q^{18})\Theta(q^3z^{-3};q^6)}\\
& \ \ \ \ \ \cdot \Big [ 
q^{-1}z^{-3}\frac{\Theta(q^4;q^6)\Theta(q^3;q^{18})}{\Theta(qz^3;q^6)} 
+qz^{-5}\frac{\Theta(q^8;q^6)\Theta(q^{15};q^{18})}{\Theta(q^5z^3;q^6)}\Big ].
\end{align*}}%
Using (\ref{equation:j-flip}) and then (\ref{equation:j-elliptic}) yields
{\allowdisplaybreaks \begin{align*}
D_{3}(qz^{-4},qz^3,z^9;q^2)
&=\frac{z^9(q^6;q^6)_{\infty}^3}
{\Theta(z;q^2)\Theta(z^9;q^{18})\Theta(q^3z^{3};q^6)}\\
& \ \ \ \ \ \cdot \Big [ 
q^{-1}z^{-3}\frac{\Theta(q^2;q^6)\Theta(q^3;q^{18})}{\Theta(qz^3;q^6)} 
+qz^{-5}\frac{\Theta(q^8;q^6)\Theta(q^{3};q^{18})}{\Theta(q^5z^3;q^6)}\Big ]\\
&=\frac{z^9(q^6;q^6)_{\infty}^3\Theta(q^{3};q^{18})}
{\Theta(z;q^2)\Theta(z^9;q^{18})\Theta(q^3z^{3};q^6)}
 \cdot \Big [ 
q^{-1}z^{-3}\frac{\Theta(q^2;q^6)}{\Theta(qz^3;q^6)} 
+qz^{-5}\frac{\Theta(q^8;q^6)}{\Theta(q^5z^3;q^6)}\Big ]\\
&=\frac{q^{-1}z^{6}(q^6;q^6)_{\infty}^3\Theta(q^2;q^6)\Theta(q^{3};q^{18})}
{\Theta(z;q^2)\Theta(z^9;q^{18})\Theta(q^3z^{3};q^6)}
 \cdot \Big [ 
\frac{1}{\Theta(qz^3;q^6)} 
-z^{-2}\frac{1}{\Theta(q^5z^3;q^6)}\Big ]\\
&=\frac{q^{-1}z^{6}(q^6;q^6)_{\infty}^3\Theta(q^2;q^6)\Theta(q^{3};q^{18})}
{\Theta(z;q^2)\Theta(z^9;q^{18})\Theta(q^3z^{3};q^6)}
 \cdot 
 \frac{\Theta(q^5z^3;q^6)-z^{-2}\Theta(qz^3;q^6)}{\Theta(qz^3;q^6)\Theta(q^5z^3;q^6)}.
\end{align*}}%
We recall the quintuple product identity and make the substitutions $q\to q^2$, $ x\to qx$.  This allows us to write
\begin{align*}
D_{3}(qz^{-4},qz^3,z^9;q^2)
&=\frac{q^{-1}z^{6}(q^6;q^6)_{\infty}^3\Theta(q^2;q^6)\Theta(q^{3};q^{18})}
{\Theta(z;q^2)\Theta(z^9;q^{18})\Theta(q^3z^{3};q^6)\Theta(qz^3;q^6)\Theta(q^5z^3;q^6)}\\
&\qquad  \cdot  \frac{(q^2;q^2)_{\infty}\Theta(q^2z^2;q^2)}{\Theta(qz;q^2)}\\
&=\frac{q^{-1}z^{6}(q^6;q^6)_{\infty}^3(q^2;q^2)_{\infty}^2\Theta(q^{3};q^{18})}
{\Theta(z;q^2)\Theta(z^9;q^{18})\Theta(q^3z^{3};q^6)\Theta(qz^3;q^6)\Theta(q^5z^3;q^6)}
 \cdot  \frac{\Theta(q^2z^2;q^2)}{\Theta(qz;q^2)}.
\end{align*}
Using (\ref{equation:j-mod}) with $n=3$ allows us to write
{\allowdisplaybreaks \begin{align*}
D_{3}(qz^{-4},qz^3,z^9;q^2)
&=\frac{q^{-1}z^{6}(q^6;q^6)_{\infty}^3(q^2;q^2)_{\infty}^2\Theta(q^{3};q^{18})}
{\Theta(z;q^2)\Theta(z^9;q^{18})}\cdot \frac{(q^2;q^2)_{\infty}}{\Theta(qz^3;q^2)(q^6;q^6)_{\infty}^3}
 \cdot  \frac{\Theta(q^2z^2;q^2)}{\Theta(qz;q^2)}\\
&=\frac{q^{-1}z^{6}(q^2;q^2)_{\infty}^3\Theta(q^{3};q^{18})}
{\Theta(z;q^2)\Theta(z^9;q^{18})}\cdot \frac{1}{\Theta(qz^3;q^2)}
 \cdot  \frac{\Theta(q^2z^2;q^2)}{\Theta(qz;q^2)}.
\end{align*}}%
Again using (\ref{equation:j-mod}) but with $n=2$ brings us to
{\allowdisplaybreaks \begin{align*}
D_{3}(qz^{-4},qz^3,z^9;q^2)
&=\frac{q^{-1}z^{6}(q^2;q^2)_{\infty}^3\Theta(q^{3};q^{18})}
{\Theta(z^9;q^{18})}\cdot \frac{\Theta(q^2z^2;q^2) }{\Theta(qz^3;q^2)}
 \cdot \frac{(q;q)_{\infty}}{\Theta(z;q)(q^2;q^2)_{\infty}^2}\\
&=\frac{q^{-1}z^{6}(q;q)_{\infty}(q^2;q^2)_{\infty}\Theta(q^{3};q^{18})}
{\Theta(z^9;q^{18})}\cdot \frac{\Theta(q^2z^2;q^2) }{\Theta(qz^3;q^2)}
 \cdot \frac{1}{\Theta(z;q)}.
\end{align*}}%
Using (\ref{equation:j-elliptic}) and then (\ref{equation:j-roots}) yields
\begin{align*}
D_{3}(qz^{-4},qz^3,z^9;q^2)
&=-\frac{q^{-1}z^{4}(q;q)_{\infty}(q^2;q^2)_{\infty}\Theta(q^{3};q^{18})}
{\Theta(z^9;q^{18})}\cdot \frac{\Theta(z^2;q^2) }{\Theta(qz^3;q^2)}
 \cdot \frac{1}{\Theta(z;q)}\\
&=-\frac{q^{-1}z^{4}(q;q)_{\infty}(q^2;q^2)_{\infty}\Theta(q^{3};q^{18})}
{\Theta(z^9;q^{18})\Theta(qz^3;q^2)\Theta(z;q)}\cdot \frac{\Theta(z;q)\Theta(-z;q)(q^2;q^2)_{\infty}}{(q;q)_{\infty}^2}\\
&=-\frac{q^{-1}z^{4}(q^2;q^2)_{\infty}^2\Theta(q^{3};q^{18})\Theta(-z;q)}
{(q;q)_{\infty}\Theta(z^9;q^{18})\Theta(qz^3;q^2)}.
\end{align*}

\subsection{Proof of (\ref{equation:idN3-6})}
We specialize (\ref{equation:msplit3}) using the first term in (\ref{equation:idN3-6}).  This gives us
\begin{align*}
D_{3}(x,q,q^3x^{-3};q^2)&=\frac{q^3x^{-3}(q^6;q^6)_{\infty}^3}{\Theta(xq;q^2)\Theta(q^3x^{-3};q^{18})\Theta(q^3;q^6)}
\Big [ 
\frac{1}{q}\frac{\Theta(q^4;q^6)\Theta(x^3;q^{18})}{\Theta(q;q^6)} \\
&\ \ \ \ \ -\frac{x}{q^2}\frac{\Theta(q^6;q^6)\Theta(q^{6}x^3;q^{18})}{\Theta(q^3;q^6)}
+\frac{x^2}{q}\frac{\Theta(q^{8};q^6)\Theta(q^{12}x^3;q^{18})}{\Theta(q^5;q^6)}\Big ]\\
&=\frac{q^3x^{-3}(q^6;q^6)_{\infty}^3}{\Theta(xq;q^2)\Theta(q^3x^{-3};q^{18})\Theta(q^3;q^6)}
\Big [ 
\frac{1}{q}\frac{\Theta(q^4;q^6)\Theta(x^3;q^{18})}{\Theta(q;q^6)} \\
&\ \ \ \ \ +\frac{x^2}{q}\frac{\Theta(q^{8};q^6)\Theta(q^{12}x^3;q^{18})}{\Theta(q^5;q^6)}\Big ].
\end{align*}
Using (\ref{equation:j-elliptic}), (\ref{equation:j-flip}), and (\ref{equation:j-elliptic}) again yields
{\allowdisplaybreaks \begin{align*}
D_{3}&(x,q,q^3x^{-3};q^2)\\
&=\frac{q^3x^{-3}(q^6;q^6)_{\infty}^3}{\Theta(xq;q^2)\Theta(q^3x^{-3};q^{18})\Theta(q^3;q^6)}
\Big [ 
\frac{1}{q}\frac{\Theta(q^4;q^6)\Theta(x^3;q^{18})}{\Theta(q;q^6)} 
 -q^{-3}x^2\frac{\Theta(q^{2};q^6)\Theta(q^{12}x^3;q^{18})}{\Theta(q^5;q^6)}\Big ]\\
&=\frac{q^3x^{-3}(q^6;q^6)_{\infty}^3}{\Theta(xq;q^2)\Theta(q^3x^{-3};q^{18})\Theta(q^3;q^6)}
\frac{\Theta(q^2;q^6)}{\Theta(q;q^6)}
 \cdot \Big [ 
q^{-1}\Theta(x^3;q^{18})
 -q^{-3}x^2\Theta(q^{12}x^3;q^{18})\Big ]\\
&= -\frac{x^{-1}(q^6;q^6)_{\infty}^3}{\Theta(xq;q^2)\Theta(q^3x^{-3};q^{18})\Theta(q^3;q^6)}
\frac{\Theta(q^2;q^6)}{\Theta(q;q^6)}
 \cdot \Big [ 
q^{2}x\Theta(q^{18}x^3;q^{18})
 +\Theta(q^{12}x^3;q^{18})\Big ].
\end{align*}}%
Using the quintuple product identity (\ref{equation:H1Thm1.0}) gives us
 \begin{align*}
D_{3}(x,q,q^3x^{-3};q^2)
&=-\frac{x^{-1}(q^6;q^6)_{\infty}^3}{\Theta(xq;q^2)\Theta(q^3x^{-3};q^{18})\Theta(q^3;q^6)}
\cdot \frac{\Theta(q^2;q^6)}{\Theta(q;q^6)}
 \cdot \frac{(q^6;q^6)_{\infty}\Theta(q^4x^2;q^6)}{\Theta(q^2x;q^6)}\\
 &=-\frac{x^{-1}(q^6;q^6)_{\infty}^3(q^2;q^2)_{\infty}^2}{\Theta(xq;q^2)\Theta(q^3x^{-3};q^{18})}
\cdot \frac{1}{(q;q)_{\infty}(q^3;q^3)_{\infty}}
 \cdot \frac{\Theta(q^4x^2;q^6)}{\Theta(q^2x;q^6)},
\end{align*}
where in the last equality we have used elementary product rearrangements.  Finally we use (\ref{equation:j-mod}) to obtain
 \begin{align*}
D_{3}(x,q,q^3x^{-3};q^2)
&=-\frac{x^{-1}(q^6;q^6)_{\infty}^3(q^2;q^2)_{\infty}^2}{\Theta(xq;q^2)\Theta(q^3x^{-3};q^{18})}
 \cdot \frac{1}{(q;q)_{\infty}(q^3;q^3)_{\infty}}
 \cdot \frac{\Theta(q^4x^2;q^6)}{\Theta(q^2x;q^6)}\frac{\Theta(x;q^2)}{\Theta(x;q^2)}\\
&=-\frac{x^{-1}(q^6;q^6)_{\infty}^3\Theta(x;q^2)}{(q^3;q^3)_{\infty}\Theta(q^3x^{-3};q^{18})}
 \cdot \frac{\Theta(q^4x^2;q^6)}{\Theta(q^2x;q^6)}\frac{1}{\Theta(x;q)}.
\end{align*}

\subsection{Proof of (\ref{equation:idN3-7})}
We specialize (\ref{equation:msplit3}) using the first term in (\ref{equation:idN3-7}).  This gives us
{\allowdisplaybreaks \begin{align*}
D_{3}(qz^{-2},z,z^3;q^2)&=\frac{z^3(q^6;q^6)_{\infty}^3}{\Theta(qz^{-1};q^2)\Theta(z^3;q^{18})\Theta(q^3z^{-3};q^6)}
\Big [ 
z^{-1}\frac{\Theta(q^3z^{-2};q^6)\Theta(1;q^{18})}{\Theta(z;q^6)} \\
&\ \ \ \ \ -z^{-2}q^{-1}\frac{\Theta(q^{5}z^{-2};q^6)\Theta(q^{6};q^{18})}{\Theta(q^2z;q^6)}
+z^{-3}\frac{\Theta(q^{7}z^{-2};q^6)\Theta(q^{12};q^{18})}{\Theta(q^4z;q^6)}\Big ]\\
&=\frac{z^3(q^6;q^6)_{\infty}^3}{\Theta(qz^{-1};q^2)\Theta(z^3;q^{18})\Theta(q^3z^{-3};q^6)}\\
& \ \ \ \ \ \cdot  \Big [ 
-z^{-2}q^{-1}\frac{\Theta(q^{5}z^{-2};q^6)\Theta(q^{6};q^{18})}{\Theta(q^2z;q^6)}
+z^{-3}\frac{\Theta(q^{7}z^{-2};q^6)\Theta(q^{12};q^{18})}{\Theta(q^4z;q^6)}\Big ].
\end{align*}}%
Using (\ref{equation:j-flip}) gives us
{\allowdisplaybreaks \begin{align*}
D_{3}(qz^{-2},z,z^3;q^2)&=\frac{z^3(q^6;q^6)_{\infty}^4}{\Theta(qz^{-1};q^2)\Theta(z^3;q^{18})\Theta(q^3z^{-3};q^6)}\\
& \ \ \ \ \ \cdot  \Big [ 
-z^{-2}q^{-1}\frac{\Theta(q^{5}z^{-2};q^6)}{\Theta(q^2z;q^6)}
+z^{-3}\frac{\Theta(q^{7}z^{-2};q^6)}{\Theta(q^4z;q^6)}\Big ]\\
&=\frac{z^3(q^6;q^6)_{\infty}^4}{\Theta(qz;q^2)\Theta(z^3;q^{18})\Theta(q^3z^{3};q^6)}\\
& \ \ \ \ \ \cdot  \Big [ 
-z^{-2}q^{-1}\frac{\Theta(qz^{2};q^6)}{\Theta(q^2z;q^6)}
+z^{-3}\frac{\Theta(q^{7}z^{-2};q^6)}{\Theta(q^4z;q^6)}\Big ].
\end{align*}}%
Using (\ref{equation:j-elliptic}) and then (\ref{equation:j-flip}) yields
{\allowdisplaybreaks \begin{align*}
D_{3}(qz^{-2},z,z^3;q^2)&=\frac{z^3(q^6;q^6)_{\infty}^4}{\Theta(qz;q^2)\Theta(z^3;q^{18})\Theta(q^3z^{3};q^6)}\\
& \ \ \ \ \ \cdot  \Big [ 
-z^{-2}q^{-1}\frac{\Theta(qz^{2};q^6)}{\Theta(q^2z;q^6)}
-q^{-1}z^{-1}\frac{\Theta(qz^{-2};q^6)}{\Theta(q^4z;q^6)}\Big ]\\
&=\frac{z^3(q^6;q^6)_{\infty}^4}{\Theta(qz;q^2)\Theta(z^3;q^{18})\Theta(q^3z^{3};q^6)}\\
& \ \ \ \ \ \cdot  \Big [ 
-z^{-2}q^{-1}\frac{\Theta(qz^{2};q^6)}{\Theta(q^2z;q^6)}
-q^{-1}z^{-1}\frac{\Theta(q^5z^{2};q^6)}{\Theta(q^4z;q^6)}\Big ]\\
&=-\frac{q^{-1}z(q^6;q^6)_{\infty}^4}{\Theta(qz;q^2)\Theta(z^3;q^{18})\Theta(q^3z^{3};q^6)}
 \cdot  \Big [ 
\frac{\Theta(qz^{2};q^6)}{\Theta(q^2z;q^6)}
+z\frac{\Theta(q^5z^{2};q^6)}{\Theta(q^4z;q^6)}\Big ].
\end{align*}}%
Combining fractions yields
\begin{align*}
D_{3}(qz^{-2},z,z^3;q^2)
&=-\frac{q^{-1}z(q^6;q^6)_{\infty}^4}{\Theta(qz;q^2)\Theta(z^3;q^{18})\Theta(q^3z^{3};q^6)}\\
& \ \ \ \ \  \cdot  \Big [ 
\frac{\Theta(qz^{2};q^6)\Theta(q^4z;q^6)
+z\Theta(q^5z^{2};q^6)\Theta(q^2z;q^6)}{\Theta(q^2z;q^6)\Theta(q^4z;q^6)}\Big ].
\end{align*}
Using (\ref{equation:WR-cor1}) gives
{\allowdisplaybreaks \begin{align*}
D_{3}(qz^{-2},z,z^3;q^2)
&=-\frac{q^{-1}z(q^6;q^6)_{\infty}^4}{\Theta(qz;q^2)\Theta(z^3;q^{18})\Theta(q^3z^{3};q^6)}\\
& \ \ \ \ \  \cdot  
\frac{1}{\Theta(q^2z;q^6)\Theta(q^4z;q^6)}\cdot \frac{\Theta(zq;q^6)\Theta(q/z;q^6)\Theta(z^2;q^6)}{\Theta(z;q^6)}\cdot \frac{\Theta(q^2;q^6)}{\Theta(q;q^6)}.
\end{align*}}%
Using (\ref{equation:j-mod}) with $n=3$ gives
{\allowdisplaybreaks \begin{align*}
D_{3}(qz^{-2},z,z^3;q^2)
&=-\frac{q^{-1}z(q^6;q^6)_{\infty}^4}{\Theta(qz;q^2)\Theta(z^3;q^{18})\Theta(q^3z^{3};q^6)}\\
& \ \ \ \ \  \cdot  
 \Theta(zq;q^6)\Theta(q/z;q^6)\Theta(z^2;q^6)\cdot \frac{\Theta(q^2;q^6)}{\Theta(q;q^6)}
 \frac{(q^2;q^2)_{\infty}}{\Theta(z;q^2)(q^6;q^6)_{\infty}^3}\\
 &=-\frac{q^{-1}z(q^6;q^6)_{\infty}}{\Theta(qz;q^2)\Theta(z^3;q^{18})\Theta(q^3z^{3};q^6)}\\
& \ \ \ \ \  \cdot  
 \Theta(zq;q^6)\Theta(q/z;q^6)\Theta(z^2;q^6)\cdot \frac{1}{\Theta(q;q^6)}\cdot 
 \frac{(q^2;q^2)_{\infty}^2}{\Theta(z;q^2)}.
\end{align*}}%
Using (\ref{equation:j-mod}) with $n=2$ yields
\begin{align*}
D_{3}(qz^{-2},z,z^3;q^2)
&=-\frac{q^{-1}z(q^6;q^6)_{\infty}}{\Theta(z^3;q^{18})\Theta(q^3z^{3};q^6)}\\
& \ \ \ \ \  \cdot  
 \Theta(zq;q^6)\Theta(q/z;q^6)\Theta(z^2;q^6)\cdot \frac{ (q^2;q^2)_{\infty}^2}{\Theta(q;q^6)}
 \cdot \frac{(q;q)_{\infty}}{\Theta(z;q)(q^2;q^2)_{\infty}^2}\\
 &=-\frac{q^{-1}z(q;q)_{\infty}(q^6;q^6)_{\infty} \Theta(zq;q^6)\Theta(q/z;q^6)\Theta(z^2;q^6)}
 {\Theta(z^3;q^{18})\Theta(q^3z^{3};q^6)\Theta(z;q)}
 \cdot \frac{ 1}{\Theta(q;q^6)}.
\end{align*}
Using (\ref{equation:j-flip}) gives
\begin{align*}
D_{3}(qz^{-2},z,z^3;q^2)
&=-\frac{q^{-1}z(q;q)_{\infty}(q^6;q^6)_{\infty} \Theta(zq;q^6)\Theta(q^5z;q^6)\Theta(z^2;q^6)}
 {\Theta(z^3;q^{18})\Theta(q^3z^{3};q^6)\Theta(z;q)}
 \cdot \frac{ 1}{\Theta(q;q^6)}.
\end{align*}
Another use of (\ref{equation:j-mod}) with $n=3$ gives
{\allowdisplaybreaks \begin{align*}
D_{3}(qz^{-2},z,z^3;q^2)
&=-\frac{q^{-1}z(q;q)_{\infty}(q^6;q^6)_{\infty} \Theta(zq;q^6)\Theta(q^5z;q^6)\Theta(z^2;q^6)}
 {\Theta(z^3;q^{18})\Theta(q^3z^{3};q^6)\Theta(z;q)}\\
&\qquad  \cdot \frac{ 1}{\Theta(q;q^6)}\cdot \frac{\Theta(q^3z;q^6)}{\Theta(q^3z;q^6)}\\
&=-\frac{q^{-1}z(q;q)_{\infty}(q^6;q^6)_{\infty} \Theta(z^2;q^6)}
 {\Theta(z^3;q^{18})\Theta(q^3z^{3};q^6)\Theta(z;q)\Theta(q;q^6)\Theta(q^3z;q^6)}
 \cdot \frac{\Theta(qz;q^2)(q^6;q^6)_{\infty}^3}{(q^2;q^2)_{\infty}}\\
&=-\frac{q^{-1}z(q;q)_{\infty}(q^6;q^6)_{\infty}^4 \Theta(z^2;q^6)}
 {\Theta(z^3;q^{18})\Theta(q^3z^{3};q^6)\Theta(z;q)\Theta(q;q^6)\Theta(q^3z;q^6)}
 \cdot \frac{\Theta(qz;q^2)}{(q^2;q^2)_{\infty}}\\
&=-\frac{q^{-1}z(q^3;q^3)_{\infty}(q^6;q^6)_{\infty}^2 \Theta(z^2;q^6)\Theta(qz;q^2)}
 {\Theta(z^3;q^{18})\Theta(q^3z^{3};q^6)\Theta(z;q)\Theta(q^3z;q^6)},
\end{align*}}%
where the last line follows from product rearrangements.

\subsection{Proof of (\ref{equation:idN3-8})}
We specialize (\ref{equation:msplit3}) using the first term in (\ref{equation:idN3-8}).  This gives us
{\allowdisplaybreaks \begin{align*}
D_{3}(x,q,q^9;q^2)
&=\frac{q^9(q^6;q^6)_{\infty}^3}{\Theta(xq;q^2)\Theta(q^9;q^{18})\Theta(x^3q^9;q^6)}
\Big [ 
\frac{1}{q}\frac{\Theta(x^3q^{10};q^6)\Theta(q^{-6};q^{18})}{\Theta(q;q^6)}\\
&\ \ \ \ \ -\frac{x}{q^2}\frac{\Theta(x^3q^{12};q^6)\Theta(q^{6}q^{-6};q^{18})}{\Theta(q^3;q^6)}
+\frac{x^2q}{q^2}\frac{\Theta(q^{14}x^3;q^6)\Theta(q^{6};q^{18})}{\Theta(q^5;q^6)}\Big ]\\
&=\frac{q^9(q^6;q^6)_{\infty}^3}{\Theta(xq;q^2)\Theta(q^9;q^{18})\Theta(x^3q^9;q^6)}
\Big [ 
\frac{1}{q}\frac{\Theta(x^3q^{10};q^6)\Theta(q^{-6};q^{18})}{\Theta(q;q^6)}\\
&\ \ \ \ \ 
+\frac{x^2q}{q^2}\frac{\Theta(q^{14}x^3;q^6)\Theta(q^{6};q^{18})}{\Theta(q^5;q^6)}\Big ].
\end{align*}}%
Using (\ref{equation:j-flip}) and (\ref{equation:j-elliptic}) yields
{\allowdisplaybreaks \begin{align*}
D_{3}(x,q,q^9;q^2)
&=\frac{q^9(q^6;q^6)_{\infty}^3}{\Theta(xq;q^2)\Theta(q^9;q^{18})\Theta(x^3q^9;q^6)}
\Big [ 
-q^{-7}\frac{\Theta(x^3q^{10};q^6)\Theta(q^{6};q^{18})}{\Theta(q;q^6)}\\
&\ \ \ \ \ 
+\frac{x^2q}{q^2}\frac{\Theta(q^{14}x^3;q^6)\Theta(q^{6};q^{18})}{\Theta(q;q^6)}\Big ]\\
&=\frac{q^9(q^6;q^6)_{\infty}^3}{\Theta(xq;q^2)\Theta(q^9;q^{18})\Theta(x^3q^9;q^6)}
\frac{\Theta(q^{6};q^{18})}{\Theta(q;q^6)}\\
& \ \ \ \ \ \cdot \Big [ 
-q^{-7}\Theta(x^3q^{10};q^6) +x^2q^{-1}\Theta(q^{14}x^3;q^6)\Big ],
\end{align*}}%
where we have pulled out a common factor.  We again employ (\ref{equation:j-elliptic}), we rewrite three of theta functions to obtain
{\allowdisplaybreaks \begin{align*}
D_{3}(x,q,q^9;q^2)
&=\frac{q^9(q^6;q^6)_{\infty}^3}{\Theta(xq;q^2)\Theta(q^9;q^{18})(-x^{-3}q^{-3})\Theta(x^3q^3;q^6)}
\frac{\Theta(q^{6};q^{18})}{\Theta(q;q^6)}\\
& \ \ \ \ \ \cdot \Big [ 
q^{-7}q^{-4}x^{-3}\Theta(x^3q^{4};q^6) +x^2q^{-1}x^{-6}q^{-10}\Theta(q^{2}x^3;q^6)\Big ]\\
&=-\frac{qx^{-1}(q^6;q^6)_{\infty}^3}{\Theta(xq;q^2)\Theta(q^9;q^{18})\Theta(x^3q^3;q^6)}
\frac{\Theta(q^{6};q^{18})}{\Theta(q;q^6)}\\
& \ \ \ \ \ \cdot \Big [ 
x\Theta(x^3q^{4};q^6) +\Theta(q^{2}x^3;q^6)\Big ].
\end{align*}}%
Using the quintuple product identity (\ref{equation:H1Thm1.0}) results in
\begin{equation*}
D_{3}(x,q,q^9;q^2)
=\frac{qx^{-1}(q^6;q^6)_{\infty}^3}{\Theta(xq;q^2)\Theta(q^9;q^{18})\Theta(x^3q^3;q^6)}
\frac{\Theta(q^{6};q^{18})}{\Theta(q;q^6)}
\frac{(q^2;q^2)_{\infty}\Theta(x^2;q^2)}{\Theta(x;q^2)}.
\end{equation*}
Using (\ref{equation:j-mod}), we have
\begin{align*}
D_{3}(x,q,q^9;q^2)
&=-\frac{qx^{-1}(q^6;q^6)_{\infty}^3}{\Theta(x;q)\Theta(q^9;q^{18})\Theta(x^3q^3;q^6)}
\frac{\Theta(q^{6};q^{18})}{\Theta(q;q^6)}
\Theta(x^2;q^2)\frac{(q;q)_{\infty}}{(q^2;q^2)_{\infty}}\\
&=-\frac{qx^{-1}(q^6;q^6)_{\infty}^4}{\Theta(x;q)\Theta(q^9;q^{18})\Theta(x^3q^3;q^6)}
\frac{(q^2;q^2)_{\infty}(q^3;q^3)_{\infty}}{(q;q)_{\infty}(q^6;q^6)_{\infty}^2}
\Theta(x^2;q^2)\frac{(q;q)_{\infty}}{(q^2;q^2)_{\infty}}\\
&=-\frac{qx^{-1}(q^6;q^6)_{\infty}^2(q^3;q^3)_{\infty}}{\Theta(x;q)\Theta(q^9;q^{18})\Theta(x^3q^3;q^6)}
\Theta(x^2;q^2),
\end{align*}
where where we have used product rearrangements and simplified.  Again using (\ref{equation:j-mod}), we have
{\allowdisplaybreaks \begin{align*}
D_{3}(x,q,q^9;q^2)
&=-\frac{qx^{-1}(q^6;q^6)_{\infty}^2(q^3;q^3)_{\infty}}{\Theta(x;q)\Theta(q^9;q^{18})\Theta(x^3q^3;q^6)}
\Theta(x^2;q^2)\frac{\Theta(qx^2;q^2)}{\Theta(qx^2;q^2)}\\
&=-\frac{qx^{-1}(q^6;q^6)_{\infty}^2(q^3;q^3)_{\infty}}{\Theta(x;q)\Theta(q^9;q^{18})\Theta(x^3q^3;q^6)}
\frac{\Theta(x^2;q)}{\Theta(qx^2;q^2)}\frac{(q^2;q^2)_{\infty}^2}{(q;q)_{\infty}}\\
&=-\frac{qx^{-1}\Theta(-q;q^4)(q^6;q^6)_{\infty}^2(q^3;q^3)_{\infty}\Theta(x^2;q)}
{\Theta(x;q)\Theta(q^9;q^{18})\Theta(x^3q^3;q^6)\Theta(qx^2;q^2)},
\end{align*}}%
where we have again used product rearrangements.

\subsection{Proof of (\ref{equation:idN3-9})}
We specialize (\ref{equation:msplit3}) using the second term in (\ref{equation:idN3-9}).  This gives us
{\allowdisplaybreaks \begin{align*}
D_{3}(qz^{-3},qz^2,q^9z^6;q^2)
&=\frac{q^9z^6(q^6;q^6)_{\infty}^3}
{\Theta(q^2z^{-1};q^2)\Theta(q^9z^6;q^{18})\Theta(q^{12}z^{-3};q^6)}
\Big [ 
\frac{1}{qz^2}\frac{\Theta(q^{13}z^{-1};q^6)\Theta(q^{-6};q^{18})}{\Theta(qz^2;q^6)} \\
&\ \ \ \ \ -\frac{qz^{-3}}{q^2}\frac{\Theta(q^{15}z^{-1};q^6)\Theta(1;q^{18})}{\Theta(q^3z^2;q^6)}
+\frac{q^3z^{-4}}{q^2}\frac{\Theta(q^{17}z^{-1};q^6)\Theta(q^{6};q^{18})}{\Theta(q^5z^2;q^6)}\Big ]\\
&=\frac{q^9z^6(q^6;q^6)_{\infty}^3}
{\Theta(q^2z^{-1};q^2)\Theta(q^9z^6;q^{18})\Theta(q^{12}z^{-3};q^6)}\\
& \ \ \ \ \ \cdot \Big [ 
q^{-1}z^{-2}\frac{\Theta(q^{13}z^{-1};q^6)\Theta(q^{-6};q^{18})}{\Theta(qz^2;q^6)} 
+qz^{-4}\frac{\Theta(q^{17}z^{-1};q^6)\Theta(q^{6};q^{18})}{\Theta(q^5z^2;q^6)}\Big ].
\end{align*}}%
Employing (\ref{equation:j-flip}) allows us to write
{\allowdisplaybreaks \begin{align*}
D_{3}(qz^{-3},qz^2,q^9z^6;q^2)
&=\frac{q^9z^6(q^6;q^6)_{\infty}^3}
{\Theta(z;q^2)\Theta(q^9z^6;q^{18})\Theta(q^{-6}z^{3};q^6)}\\
& \ \ \ \ \ \cdot \Big [ 
q^{-1}z^{-2}\frac{\Theta(q^{13}z^{-1};q^6)\Theta(q^{24};q^{18})}{\Theta(qz^2;q^6)} 
+qz^{-4}\frac{\Theta(q^{17}z^{-1};q^6)\Theta(q^{6};q^{18})}{\Theta(q^5z^2;q^6)}\Big ].
\end{align*}}%
Using (\ref{equation:j-elliptic}) gives us
{\allowdisplaybreaks \begin{align*}
D_{3}(qz^{-3},qz^2,q^9z^6;q^2)
&=\frac{q^{15}z^3(q^6;q^6)_{\infty}^3}
{\Theta(z;q^2)\Theta(q^9z^6;q^{18})\Theta(z^{3};q^6)}\\
& \qquad \cdot \Big [ 
-q^{-15}\frac{\Theta(qz^{-1};q^6)\Theta(q^{6};q^{18})}{\Theta(qz^2;q^6)} 
+q^{-15}z^{-2}\frac{\Theta(q^{5}z^{-1};q^6)\Theta(q^{6};q^{18})}{\Theta(q^5z^2;q^6)}\Big ]\\
&=-\frac{z(q^6;q^6)_{\infty}^4}
{\Theta(z;q^2)\Theta(q^9z^6;q^{18})\Theta(z^{3};q^6)}
 \cdot \Big [ 
\frac{-z^2\Theta(qz^{-1};q^6)}{\Theta(qz^2;q^6)} 
+\frac{\Theta(q^{5}z^{-1};q^6)}{\Theta(q^5z^2;q^6)}\Big ]\\
&=-\frac{z(q^6;q^6)_{\infty}^4}
{\Theta(z;q^2)\Theta(q^9z^6;q^{18})\Theta(z^{3};q^6)}\\
& \ \ \ \ \  \cdot \Big [ 
\frac{\Theta(qz^2;q^6)\Theta(qz;q^6)-z^2\Theta(q^5z^2;q^6)\Theta(q^5z;q^6)}{\Theta(qz^2;q^6)\Theta(q^5z^2;q^6)}\Big ].
\end{align*}}%
Using (\ref{equation:WR-cor2}) gives us
\begin{align*}
D_{3}(qz^{-3},qz^2,q^9z^6;q^2)
&=-\frac{z(q^6;q^6)_{\infty}^4}
{\Theta(z;q^2)\Theta(q^9z^6;q^{18})\Theta(z^{3};q^6)}\\
& \ \ \ \ \  \cdot 
\frac{1}{\Theta(qz^2;q^6)\Theta(q^5z^2;q^6)}
\cdot \frac{\Theta(zq^2;q^6)\Theta(q^2/z;q^6)\Theta(z^2;q^6)}{\Theta(zq^3;q^6)}\cdot \frac{\Theta(q^2;q^6)}{\Theta(q;q^6)}.
\end{align*}
We use (\ref{equation:j-flip}) and the (\ref{equation:j-mod}) with $n=3$ and $n=2$ to get
{\allowdisplaybreaks \begin{align*}
D_{3}(qz^{-3},qz^2,q^9z^6;q^2)
&=-\frac{z(q^6;q^6)_{\infty}^4}
{\Theta(z;q^2)\Theta(q^9z^6;q^{18})\Theta(z^{3};q^6)}\cdot 
\frac{1}{\Theta(qz^2;q^6)\Theta(q^5z^2;q^6)}\cdot \frac{\Theta(q^3z^2;q^6)}{\Theta(q^3z^2;q^6)}\\
& \qquad 
\cdot \frac{\Theta(zq^2;q^6)\Theta(q^4z;q^6)\Theta(z^2;q^6)}{\Theta(zq^3;q^6)}\cdot \frac{\Theta(q^2;q^6)}{\Theta(q;q^6)}\\
&=-\frac{z(q^6;q^6)_{\infty}}
{\Theta(z;q^2)\Theta(q^9z^6;q^{18})\Theta(z^{3};q^6)}\cdot 
\frac{(q^2;q^2)_{\infty}^2}{\Theta(qz^2;q^2)}\\
& \qquad 
\cdot \frac{\Theta(zq^2;q^6)\Theta(q^4z;q^6)\Theta(z^2;q^6)}{\Theta(zq^3;q^6)}
\cdot \frac{ \Theta(q^3z^2;q^6)}{\Theta(q;q^6)}\\
&=-\frac{z(q^6;q^6)_{\infty}}
{\Theta(z;q^2)\Theta(q^9z^6;q^{18})\Theta(z^{3};q^6)}\cdot 
\frac{(q^2;q^2)_{\infty}^2}{\Theta(qz^2;q^2)}\\
& \ \ \ \ \  
\cdot \frac{\Theta(zq^2;q^6)\Theta(q^4z;q^6)}{\Theta(zq^3;q^6)}\cdot \frac{1}{\Theta(q;q^6)}
\cdot \frac{\Theta(z^2;q^3)(q^6;q^6)_{\infty}^2}{(q^3;q^3)_{\infty}}.
\end{align*}}%
We then use (\ref{equation:j-mod}) with $n=3$ and then simplify to obtain
{\allowdisplaybreaks \begin{align*}
D_{3}(qz^{-3},qz^2,q^9z^6;q^2)
&=-\frac{z(q^6;q^6)_{\infty}^3}
{\Theta(z;q^2)\Theta(q^9z^6;q^{18})\Theta(z^{3};q^6)}\cdot 
\frac{(q^2;q^2)_{\infty}^2}{\Theta(qz^2;q^2)}\\
& \ \ \ \ \  
\cdot \frac{\Theta(zq^2;q^6)\Theta(q^4z;q^6)}{\Theta(zq^3;q^6)}\cdot \frac{\Theta(z;q^6)}{\Theta(z;q^6)}
\cdot \frac{1}{\Theta(q;q^6)}
\cdot \frac{\Theta(z^2;q^3)}{(q^3;q^3)_{\infty}}\\
&=-\frac{z(q^6;q^6)_{\infty}^3}
{\Theta(z;q^2)\Theta(q^9z^6;q^{18})\Theta(z^{3};q^6)}\cdot 
\frac{(q^2;q^2)_{\infty}^2}{\Theta(qz^2;q^2)}\\
& \ \ \ \ \  
\cdot \frac{1}{\Theta(zq^3;q^6)\Theta(z;q^6)}\cdot \frac{\Theta(z;q^2)(q^6;q^6)_{\infty}^3}{(q^2;q^2)_{\infty}}
\cdot \frac{1}{\Theta(q;q^6)}
\cdot \frac{\Theta(z^2;q^3)}{(q^3;q^3)_{\infty}}\\
&=-\frac{z(q^6;q^6)_{\infty}^6}
{\Theta(q^9z^6;q^{18})\Theta(z^{3};q^6)}\cdot 
\frac{(q^2;q^2)_{\infty}}{\Theta(qz^2;q^2)}\\
& \ \ \ \ \  
\cdot \frac{1}{\Theta(zq^3;q^6)\Theta(z;q^6)}
\cdot \frac{1}{\Theta(q;q^6)}
\cdot \frac{\Theta(z^2;q^3)}{(q^3;q^3)_{\infty}}.
\end{align*}}%
Again using (\ref{equation:j-mod}) with $n=2$ yields
{\allowdisplaybreaks \begin{align*}
D_{3}(qz^{-3},qz^2,q^9z^6;q^2)
&=-\frac{z(q^6;q^6)_{\infty}^6}
{\Theta(q^9z^6;q^{18})\Theta(z^{3};q^6)}\cdot 
\frac{(q^2;q^2)_{\infty}}{\Theta(qz^2;q^2)}\\
& \ \ \ \ \  
\cdot \frac{(q^3;q^3)_{\infty}}{\Theta(z;q^3)(q^6;q^6)_{\infty}^2}
\cdot \frac{1}{\Theta(q;q^6)}
\cdot \frac{\Theta(z^2;q^3)}{(q^3;q^3)_{\infty}}\\
&=-\frac{z(q^6;q^6)_{\infty}^4}
{\Theta(q^9z^6;q^{18})\Theta(z^{3};q^6)}\cdot 
\frac{(q^2;q^2)_{\infty}\Theta(z^2;q^3)}{\Theta(qz^2;q^2)\Theta(z;q^3)\Theta(q;q^6)}\\
&=-\frac{z(q^6;q^6)_{\infty}^2}
{\Theta(q^9z^6;q^{18})\Theta(z^{3};q^6)}\cdot 
\frac{(q^2;q^2)_{\infty}^2\Theta(z^2;q^3)}{\Theta(qz^2;q^2)\Theta(z;q^3)}
\cdot\frac{(q^3;q^3)_{\infty}}{(q;q)_{\infty}}\\
&=-\frac{z(q^6;q^6)_{\infty}^2(q^3;q^3)_{\infty}}
{\Theta(q^9z^6;q^{18})\Theta(z^{3};q^6)}\cdot 
\frac{\Theta(-q;q^4)\Theta(z^2;q^3)}{\Theta(qz^2;q^2)\Theta(z;q^3)},
\end{align*}}%
where in the last two lines we have used elementary product rearrangements.

\subsection{Proof of (\ref{equation:idN3-10})}
We specialize (\ref{equation:msplit3}) using the first term in (\ref{equation:idN3-10}).  This gives us
{\allowdisplaybreaks \begin{align*}
D_{3}(qu^{-5},u^3,u^9;q^2)&=\frac{u^9(q^6;q^6)_{\infty}^3}{\Theta(qu^{-2};q^2)\Theta(u^9;q^{18})\Theta(q^3u^{-6};q^6)}
\Big [ 
\frac{1}{u^3}\frac{\Theta(q^{3}u^{-3};q^6)\Theta(1;q^{18})}{\Theta(u^3;q^6)} \\
&\ \ \ \ \ -\frac{qu^{-5}}{q^2}\frac{\Theta(q^5u^{-3};q^6)\Theta(q^{6};q^{18})}{\Theta(q^2u^3;q^6)}
+\frac{q^{2}u^{-7}}{q^2}\frac{\Theta(q^7u^{-3};q^6)\Theta(q^{12};q^{18})}{\Theta(q^4u^3;q^6)}\Big ]\\
&=\frac{u^9(q^6;q^6)_{\infty}^3}{\Theta(qu^{-2};q^2)\Theta(u^9;q^{18})\Theta(q^3u^{-6};q^6)}\\
& \ \ \ \ \ \cdot \Big [  -u^{-5}q^{-1}\frac{\Theta(q^5u^{-3};q^6)\Theta(q^{6};q^{18})}{\Theta(q^2u^3;q^6)}
+u^{-7}\frac{\Theta(q^7u^{-3};q^6)\Theta(q^{12};q^{18})}{\Theta(q^4u^3;q^6)}\Big ].
\end{align*}}%
Using (\ref{equation:j-flip}) and pulling out a common factor gives
{\allowdisplaybreaks \begin{align*}
D_{3}(qu^{-5},u^3,u^9;q^2)
&=\frac{u^9(q^6;q^6)_{\infty}^4}{\Theta(qu^{2};q^2)\Theta(u^9;q^{18})\Theta(q^3u^{6};q^6)}\\
& \ \ \ \ \ \cdot \Big [ 
 -u^{-5}q^{-1}\frac{\Theta(qu^{3};q^6)}{\Theta(q^2u^3;q^6)}
+u^{-7}\frac{\Theta(q^7u^{-3};q^6)}{\Theta(q^4u^3;q^6)}\Big ].
\end{align*}}%
We use that (\ref{equation:j-elliptic}) and (\ref{equation:j-flip}) gives
{\allowdisplaybreaks \begin{align*}
D_{3}(qu^{-5},u^3,u^9;q^2)
&=\frac{u^9(q^6;q^6)_{\infty}^4}{\Theta(qu^{2};q^2)\Theta(u^9;q^{18})\Theta(q^3u^{6};q^6)}\\
& \ \ \ \ \ \cdot \Big [ 
 -u^{-5}q^{-1}\frac{\Theta(qu^{3};q^6)}{\Theta(q^2u^3;q^6)}
-q^{-1}u^{-4}\frac{\Theta(qu^{-3};q^6)}{\Theta(q^4u^3;q^6)}\Big ]\\
&=\frac{u^9(q^6;q^6)_{\infty}^4}{\Theta(qu^{2};q^2)\Theta(u^9;q^{18})\Theta(q^3u^{6};q^6)}\\
& \ \ \ \ \ \cdot \Big [ 
-u^{-5}q^{-1}\frac{\Theta(qu^{3};q^6)}{\Theta(q^2u^3;q^6)}
-q^{-1}u^{-4}\frac{\Theta(q^5u^{3};q^6)}{\Theta(q^4u^3;q^6)}\Big ]\\
&=-\frac{q^{-1}u^4(q^6;q^6)_{\infty}^4}{\Theta(qu^{2};q^2)\Theta(u^9;q^{18})\Theta(q^3u^{6};q^6)}
\Big [  \frac{\Theta(qu^{3};q^6)}{\Theta(q^2u^3;q^6)}
+u\frac{\Theta(q^5u^{3};q^6)}{\Theta(q^4u^3;q^6)}\Big ].
\end{align*}}%
Combining fractions and using (\ref{equation:j-mod}) with $n=2$ gives
{\allowdisplaybreaks \begin{align*}
D_{3}(qu^{-5},u^3,u^9;q^2)
&=-\frac{q^{-1}u^4(q^6;q^6)_{\infty}^4}{\Theta(qu^{2};q^2)\Theta(u^9;q^{18})\Theta(q^3u^{6};q^6)}\\
& \ \ \ \ \ \cdot \Big [  \frac{\Theta(qu^{3};q^6)\Theta(q^4u^3;q^6)
+u\Theta(q^2u^{3};q^6)\Theta(q^5u^3;q^6)}{\Theta(q^2u^3;q^6)\Theta(q^4u^3;q^6)}\Big ]\\
&=-\frac{q^{-1}u^4(q^6;q^6)_{\infty}^4}{\Theta(qu^{2};q^2)\Theta(u^9;q^{18})\Theta(q^3u^{6};q^6)}\\
& \ \ \ \ \ \cdot \frac{(q^6;q^6)_{\infty}^2}{(q^3;q^3)_{\infty}}\Big [  \frac{\Theta(qu^{3};q^3)
+u\Theta(q^2u^{3};q^3)}{\Theta(q^2u^3;q^6)\Theta(q^4u^3;q^6)}\Big ]\\
&=-\frac{q^{-1}u^4(q^6;q^6)_{\infty}^6}
{(q^3;q^3)_{\infty}\Theta(qu^{2};q^2)\Theta(u^9;q^{18})\Theta(q^3u^{6};q^6)\Theta(q^2u^3;q^6)\Theta(q^4u^3;q^6)}\\
& \ \ \ \ \ \cdot \Big [  \Theta(qu^{3};q^3)
+u\Theta(q^2u^{3};q^3)\Big ].
\end{align*}}%
Using the quintuple product identity (\ref{equation:H1Thm1.0}) yields
{\allowdisplaybreaks \begin{align*}
D_{3}(qu^{-5},u^3,u^9;q^2)
&=-\frac{q^{-1}u^4(q^6;q^6)_{\infty}^6}
{(q^3;q^3)_{\infty}\Theta(qu^{2};q^2)\Theta(u^9;q^{18})\Theta(q^3u^{6};q^6)\Theta(q^2u^3;q^6)\Theta(q^4u^3;q^6)}\\
& \ \ \ \ \ \cdot \frac{(q;q)_{\infty}\Theta(u^2;q)}{\Theta(u;q)}.\qedhere
\end{align*}}%

\subsection{Proof of (\ref{equation:idN3-11})}
We specialize (\ref{equation:msplit3}) using the second term in (\ref{equation:idN3-11}).  This gives us
{\allowdisplaybreaks \begin{align*}
D_3(qz^{-4},qz^3,q^9z^9;q^2)
&=\frac{q^9z^9(q^6;q^6)_{\infty}^3}{\Theta(q^2z^{-1};q^2)\Theta(q^9z^9;q^{18})\Theta(q^{12}z^{-3};q^6)}
\Big [ 
\frac{1}{qz^3}\frac{\Theta(q^{13};q^6)\Theta(q^{-6};q^{18})}{\Theta(qz^3;q^6)}\\
&\ \ \ \ \ -\frac{qz^{-4}}{q^2}\frac{\Theta(q^{15};q^6)\Theta(1;q^{18})}{\Theta(q^3z^3;q^6)}
+\frac{x^2z}{q^2}\frac{\Theta(q^{17};q^6)\Theta(q^{6};q^{18})}{\Theta(q^5z^3;q^6)}\Big ]\\
&=\frac{q^9z^9(q^6;q^6)_{\infty}^3}{\Theta(q^2z^{-1};q^2)\Theta(q^9z^9;q^{18})\Theta(q^{12}z^{-3};q^6)}\\
& \ \ \ \ \ \cdot \Big [ 
q^{-1}z^{-3}\frac{\Theta(q^{13};q^6)\Theta(q^{-6};q^{18})}{\Theta(qz^3;q^6)}
+qz^{-5}\frac{\Theta(q^{17};q^6)\Theta(q^{6};q^{18})}{\Theta(q^5z^3;q^6)}\Big ].
\end{align*}}%
Using (\ref{equation:j-mod}) and (\ref{equation:j-elliptic}) gives
{\allowdisplaybreaks \begin{align*}
D_3(qz^{-4},qz^3,q^9z^9;q^2)
&=\frac{q^9z^9(q^6;q^6)_{\infty}^3}
{\Theta(z;q^2)\Theta(q^9z^9;q^{18})\Theta(q^{-6}z^{3};q^6)}\\
& \ \ \ \ \ \cdot \Big [ 
q^{-1}z^{-3}\frac{\Theta(q^{13};q^6)\Theta(q^{24};q^{18})}{\Theta(qz^3;q^6)}
+qz^{-5}\frac{\Theta(q^{17};q^6)\Theta(q^{6};q^{18})}{\Theta(q^5z^3;q^6)}\Big ]\\
&=-\frac{q^{15}z^6(q^6;q^6)_{\infty}^3}
{\Theta(z;q^2)\Theta(q^9z^9;q^{18})\Theta(z^{3};q^6)}\\
& \ \ \ \ \ \cdot \Big [ 
-q^{-15}z^{-3}\frac{\Theta(q;q^6)\Theta(q^{6};q^{18})}{\Theta(qz^3;q^6)}
+q^{-15}z^{-5}\frac{\Theta(q^{5};q^6)\Theta(q^{6};q^{18})}{\Theta(q^5z^3;q^6)}\Big ]\\
&=-\frac{z(q^6;q^6)_{\infty}^4}
{\Theta(z;q^2)\Theta(q^9z^9;q^{18})\Theta(z^{3};q^6)}
\cdot \Big [ 
-z^{2}\frac{\Theta(q;q^6)}{\Theta(qz^3;q^6)}
+\frac{\Theta(q^{5};q^6)}{\Theta(q^5z^3;q^6)}\Big ].
\end{align*}}%
Again using (\ref{equation:j-flip}) to pull out a common factor and then combining fractions yields
\begin{align*}
D_3(qz^{-4},qz^3,q^9z^9;q^2)
&=\frac{z^3(q^6;q^6)_{\infty}^4\Theta(q;q^6)}
{\Theta(z;q^2)\Theta(q^9z^9;q^{18})\Theta(z^{3};q^6)}
 \cdot 
\frac{\Theta(q^5z^3;q^6)-z^{-2}\Theta(qz^3;q^6)}
{\Theta(qz^3;q^6)\Theta(q^5z^3;q^6)}.
\end{align*}
We use the quintuple product identity with $q\to q^2$, $ x\to qx$ to obtain
\begin{align*}
D_3(qz^{-4},qz^3,q^9z^9;q^2)
&=\frac{z^3(q^6;q^6)_{\infty}^4\Theta(q;q^6)}
{\Theta(q^9z^9;q^{18})\Theta(z^{3};q^6)\Theta(qz^3;q^6)\Theta(q^5z^3;q^6)}\\
& \qquad   \cdot \frac{1}{\Theta(z;q^2)}\cdot 
\frac{(q^2;q^2)_{\infty}\Theta(q^2z^2;q^2)}{\Theta(qz;q^2)}.
\end{align*}
Using (\ref{equation:j-mod}) with $n=2$ and then (\ref{equation:j-elliptic}) yields
{\allowdisplaybreaks \begin{align*}
D_3(qz^{-4},qz^3,q^9z^9;q^2)
&=\frac{z^3(q^6;q^6)_{\infty}^4\Theta(q;q^6)}
{\Theta(q^9z^9;q^{18})\Theta(z^{3};q^6)\Theta(qz^3;q^6)\Theta(q^5z^3;q^6)}\\
& \ \ \ \ \  \cdot 
(q^2;q^2)_{\infty}\Theta(q^2z^2;q^2) \cdot \frac{(q;q)_{\infty}}{\Theta(z;q)(q^2;q^2)_{\infty}^2}\\
&=\frac{z^3(q^6;q^6)_{\infty}^4\Theta(q;q^6)}
{\Theta(q^9z^9;q^{18})\Theta(z^{3};q^6)\Theta(qz^3;q^6)\Theta(q^5z^3;q^6)}
  \cdot \frac{\Theta(q^2z^2;q^2)(q;q)_{\infty}}{\Theta(z;q)(q^2;q^2)_{\infty}}\\
&=-\frac{z(q^6;q^6)_{\infty}^4\Theta(q;q^6)}
{\Theta(q^9z^9;q^{18})\Theta(z^{3};q^6)\Theta(qz^3;q^6)\Theta(q^5z^3;q^6)}
  \cdot \frac{\Theta(z^2;q^2)(q;q)_{\infty}}{\Theta(z;q)(q^2;q^2)_{\infty}}.
\end{align*}}%
Again using (\ref{equation:j-mod}) with $n=2 $ and elementary product rearrangements brings us to
{\allowdisplaybreaks \begin{align*}
D_3&(qz^{-4},qz^3,q^9z^9;q^2)\\
&=-\frac{z(q^6;q^6)_{\infty}^4\Theta(q;q^6)}
{\Theta(q^9z^9;q^{18})\Theta(z^{3};q^6)\Theta(qz^3;q^6)\Theta(q^5z^3;q^6)}\\
& \ \ \ \ \   \cdot \frac{(q;q)_{\infty}}{\Theta(z;q)(q^2;q^2)_{\infty}}
\cdot \frac{1}{\Theta(qz^2;q^2)}
\cdot \frac{\Theta(z^2;q)(q^2;q^2)_{\infty}^2}{(q;q)_{\infty}}\\
&=-\frac{z(q;q)_{\infty}(q^6;q^6)_{\infty}^6\Theta(z^2;q)}
{(q^3;q^3)_{\infty}\Theta(z;q)\Theta(qz^2;q^2)\Theta(q^9z^9;q^{18})\Theta(z^{3};q^6)\Theta(qz^3;q^6)\Theta(q^5z^3;q^6)}.\qedhere
\end{align*}}%

\section{Proof of Theorem \ref{theorem:caseN4}}\label{section:N4-proofs}
 \subsection{Proof of (\ref{equation:idN4-1})}  The first equality in the equation follows from (\ref{equation:Dn-funcEqn1}).
We specialize (\ref{equation:msplit4}) using the second term in (\ref{equation:idN4-1}).  This gives us
{\allowdisplaybreaks \begin{align*}
D_4(-z^{-5},-z^4,z^{16};q)
&=\frac{z^{16} (q^4;q^4)_{\infty}^3}{\Theta(z^{-1};q) \Theta(z^{16};q^{16}) \Theta\big(-q^{6} z^{-4};q^4)} \Big [ 
% r=0
\frac{\Theta\big(q^{6} ;q^4\big)
\Theta( 1;q^{16})}
{\Theta( -z^{4};q^4\big )}\\
% r=1
& \qquad  -\frac{z^{-1}
\Theta\big(q^{7};q^4\big)
\Theta(q^{4} ;q^{16})}
{\Theta(-q z^{4};q^4\big )}
+\frac{q z^{-2}
\Theta\big(q^{8};q^4\big)
\Theta(q^{8} ;q^{16})}
{\Theta(-q^2 z^4;q^4\big )}\\
% r=3
&\qquad -\frac{q^{3}z^{-3}
\Theta\big(q^{9};q^4\big)
\Theta(q^{12} ;q^{16})}
{\Theta(-q^3 z^{4};q^4\big )}\Big ]\\
&=\frac{z^{16} (q^4;q^4)_{\infty}^3}{\Theta(z^{-1};q) \Theta(z^{16};q^{16}) \Theta\big(-q^{6} z^{-4};q^4)} \\
& \qquad \cdot \Big [ 
 -\frac{z^{-1}
\Theta\big(q^{7};q^4\big)
\Theta(q^{4} ;q^{16})}
{\Theta(-q z^{4};q^4\big )}
 -\frac{q^{3}z^{-3}
\Theta\big(q^{9};q^4\big)
\Theta(q^{12} ;q^{16})}
{\Theta(-q^3 z^{4};q^4\big )}\Big ].
\end{align*}}%
Using (\ref{equation:j-elliptic}) and then (\ref{equation:j-flip}) yields
{\allowdisplaybreaks \begin{align*}
D_4(-z^{-5},-z^4,z^{16};q)
&=\frac{q^2z^{12} (q^4;q^4)_{\infty}^3}
{\Theta(z^{-1};q) \Theta(z^{16};q^{16}) \Theta\big(-q^{2} z^{-4};q^4)} \\
& \ \ \ \ \ \cdot \Big [ 
 \frac{z^{-1}q^{-3}\Theta\big(q^{3};q^4\big)
\Theta(q^{4} ;q^{16})}
{\Theta(-q z^{4};q^4\big )}
 -\frac{q^{-3}z^{-3}
\Theta\big(q;q^4\big)
\Theta(q^{12} ;q^{16})}
{\Theta(-q^3 z^{4};q^4\big )}\Big ]\\
&=\frac{q^2z^{12} (q^4;q^4)_{\infty}^3\Theta(q;q^4)\Theta(q^4;q^{16})}
{\Theta(qz;q) \Theta(z^{16};q^{16}) \Theta\big(-q^{2} z^{4};q^4)} 
\cdot \Big [ 
 \frac{z^{-1}q^{-3}}
{\Theta(-q z^{4};q^4\big )}
 -\frac{q^{-3}z^{-3}}
{\Theta(-q^3 z^{4};q^4\big )}\Big ]\\
&=-\frac{q^{-1}z^{9} (q^4;q^4)_{\infty}^3\Theta(q;q^4)\Theta(q^4;q^{16})}
{\Theta(qz;q) \Theta(z^{16};q^{16}) \Theta\big(-q^{2} z^{4};q^4)} 
\cdot 
 \frac{\Theta(-q z^{4};q^4\big )-z^2\Theta(-q^3 z^{4};q^4\big )}
{\Theta(-q z^{4};q^4\big )\Theta(-q^3 z^{4};q^4\big )}.
\end{align*}}%
Again using (\ref{equation:j-elliptic}) produces
\begin{align*}
D_4(-z^{-5},-z^4,z^{16};q)
&=\frac{q^{-1}z^{10} (q^4;q^4)_{\infty}^3\Theta(q;q^4)\Theta(q^4;q^{16})}
{\Theta(z;q) \Theta(z^{16};q^{16}) \Theta\big(-q^{2} z^{4};q^4)} 
\cdot 
 \frac{\Theta(-q z^{4};q^4\big )-z^2\Theta(-q^3 z^{4};q^4\big )}
{\Theta(-q z^{4};q^4\big )\Theta(-q^3 z^{4};q^4\big )}.
\end{align*}
Using (\ref{equation:j-split}) with $m=2$ brings us to
\begin{equation*}
D_4(-z^{-5},-z^4,z^{16};q)
=\frac{q^{-1}z^{10} (q^4;q^4)_{\infty}^3\Theta(q;q^4)\Theta(q^4;q^{16})}
{\Theta(z;q) \Theta(z^{16};q^{16}) \Theta\big(-q^{2} z^{4};q^4)} 
\cdot 
 \frac{\Theta(z^2;q)}
{\Theta(-q z^{4};q^4\big )\Theta(-q^3 z^{4};q^4\big )}.
\end{equation*}
Again employing (\ref{equation:j-mod}) with $n=2$ and elementary product rearrangements yields
{\allowdisplaybreaks \begin{align*}
D_4(-z^{-5},-z^4,z^{16};q)
&=\frac{q^{-1}z^{10} (q^4;q^4)_{\infty}^3\Theta(q;q^4)\Theta(q^4;q^{16})\Theta(z^2;q)}
{\Theta(z;q) \Theta(z^{16};q^{16}) \Theta\big(-q^{2} z^{4};q^4)} 
\cdot 
 \frac{(q^2;q^2)_{\infty}}{\Theta(-qz^4;q^2)(q^4;q^4)_{\infty}^2}\\
&=\frac{q^{-1}z^{10} (q^4;q^4)_{\infty}^3\Theta(z^2;q)}
{\Theta(z;q) \Theta(z^{16};q^{16}) \Theta\big(-q^{2} z^{4};q^4)} \\
& \ \ \ \ \ \cdot  \frac{(q;q)_{\infty}(q^4;q^4)_{\infty}}{(q^2;q^2)_{\infty}}
\cdot  \frac{(q^4;q^4)_{\infty}(q^{16};q^{16})_{\infty}}{(q^8;q^8)_{\infty}}
\cdot \frac{(q^2;q^2)_{\infty}}{\Theta(-qz^4;q^2)(q^4;q^4)_{\infty}^2}\\
&=\frac{q^{-1}z^{10}  (q;q)_{\infty}(q^4;q^4)_{\infty}^3(q^{16};q^{16})_{\infty}\Theta(z^2;q)}
{(q^8;q^8)_{\infty}\Theta(z;q) \Theta(z^{16};q^{16}) \Theta\big(-q^{2} z^{4};q^4)\Theta(-qz^4;q^2)}.\qedhere
\end{align*}}%

\section{Statement of results: additional theta function identities}\label{section:ThetaIds}
Note that there are several ways to write the left-hand and right-hand sides of the identities.  Sometimes we choose an expression for the sake of compactness. 

\begin{theorem}\label{theorem:caseN3-ThetaA}   For identities (\ref{equation:idN3-1}) to (\ref{equation:idN3-5}) we have the following theta function identities:
{\allowdisplaybreaks \begin{align}
&
z\frac{\Theta(z^{2};q^3)\Theta(q^{3}z^{6};q^{9})}{\Theta(z;q^3)} +q\frac{\Theta(q^2z^{2};q^3)\Theta(z^{6};q^{9})}{\Theta(qz;q^3)}
 +z^{3}\frac{\Theta(qz^{2};q^3)\Theta(q^{6}z^{6};q^{9})}{\Theta(q^2z;q^3)}\label{equation:idN3-1Theta}\\
& \qquad \qquad =\frac{z(q;q)_{\infty}(q^3;q^3)_{\infty}\Theta(z^3;q)}{\Theta(z;q)\Theta(q^2z^3;q^3)},\notag\\
&z^{4}\frac{\Theta(z^{2};q^3)\Theta(q^{6}z^{6};q^{9})}{\Theta(z;q^3)}
 +z^{3}\frac{\Theta(q^2z^{2};q^3)\Theta(q^{3}z^{6};q^{9})}{\Theta(qz;q^3)}
-q\frac{\Theta(qz^{2};q^3)\Theta(z^{6};q^{9})}{\Theta(q^2z;q^3)}\label{equation:idN3-2Theta}\\
& \qquad \qquad  =\frac{z^3(q;q)_{\infty}(q^3;q^3)_{\infty} \Theta(z^{3};q)}{\Theta(z;q)\Theta(qz^3;q^3)},\notag\\
&
\frac{\Theta(u^{4};q^3)\Theta(q^{3}u^{3};q^{9})}{\Theta(u^2;q^3)} 
 +uq\frac{\Theta(q^2u^{4};q^3)\Theta(u^{3};q^{9})}{\Theta(qu^2;q^3)}
+u\frac{\Theta(qu^{4};q^3)\Theta(q^{6}u^{3};q^{9})}{\Theta(q^2u^2;q^3)}\label{equation:idN3-3Theta}\\
& \qquad \qquad
= \frac{(q;q)_{\infty}(q^3;q^3)_{\infty}}{\Theta(u;q)\Theta(q^2u^3;q^3)}
\cdot \Theta(u^{3};q),\notag\\
&u^{2}\frac{\Theta(u^{4};q^3)\Theta(q^{6}u^{3};q^{9})}{\Theta(u^2;q^3)} 
 +u^3\frac{\Theta(q^2u^{4};q^3)\Theta(q^{3}u^{3};q^{9})}{\Theta(qu^2;q^3)}
-q\frac{\Theta(qu^{4};q^3)\Theta(u^{3};q^{9})}{\Theta(q^2u^2;q^3)}\label{equation:idN3-4Theta}\\
&\qquad \qquad=\frac{u^2(q;q)_{\infty}(q^3;q^3)_{\infty}\Theta(u^{3};q)}{\Theta(u;q)\Theta(qu^3;q^3)},\notag\\
&q\frac{\Theta(q^3z^{2};q^6)\Theta(z^{6};q^{18})}{\Theta(z;q^6)}
 +z^3\frac{\Theta(qz^{2};q^6)\Theta(q^{12}z^{6};q^{18})}{\Theta(q^2z;q^6)}
+z^2\frac{\Theta(q^5z^{2};q^6)\Theta(q^{6}z^{6};q^{18})}{\Theta(q^4z;q^6)}\label{equation:idN3-5Theta}\\
&\qquad \qquad =z^2\frac{(q^2;q^2)_{\infty}^2\Theta(q^3;q^{18})\Theta(-z;q)}{(q;q)_{\infty}}
\frac{\Theta(q^3z^{3};q^6)}{(q^6;q^6)_{\infty}^3}.\notag
\end{align}}%
\end{theorem}
Note that we do not have an identity for (\ref{equation:idN3-6}).

\begin{theorem}\label{theorem:caseN3-ThetaB} For identities (\ref{equation:idN3-7}) to (\ref{equation:idN3-11}) we have the following theta function identities:
{\allowdisplaybreaks \begin{align}
&z\frac{\Theta(q^2z^{2};q^6)\Theta(q^3;q^{18})}{\Theta(qz;q^6)} 
-\frac{\Theta(z^{2};q^6)\Theta(q^{9};q^{18})}{\Theta(q^3z;q^6)}
-z\frac{\Theta(q^4z^{2};q^6)\Theta(q^{3};q^{18})}{\Theta(q^5z;q^6)}\label{equation:idN3-7Theta}\\
&\qquad \qquad = -\frac{(q^3;q^3)_{\infty}\Theta(z^2;q^6)}
{\Theta(q^3z;q^6)(q^6;q^6)_{\infty}}
\frac{(q^2;q^2)_{\infty}^2}{(q;q)_{\infty}},\notag\\
& \frac{\Theta(q^{4}x^2;q^6)\Theta(q^{6}x^{3};q^{18})}{\Theta(q^5x;q^6)} 
 +x^{-2}q^{3}\frac{\Theta(x^2;q^6)\Theta(x^{3};q^{18})}{\Theta(q^3x;q^6)}
+x\frac{\Theta(q^{2}x^2;q^6)\Theta(q^{12}x^{3};q^{18})}{\Theta(qx;q^6)} \label{equation:idN3-8Theta}\\
&\qquad \qquad =\frac{(q^3;q^3)_{\infty}\Theta(-x;q)}
{(q^6;q^6)_{\infty}},\notag\\
&\frac{\Theta(z^{2};q^6)\Theta(q^{9}z^{3};q^{18})}{\Theta(z;q^6)} 
+ q\frac{\Theta(q^{4}z^{2};q^6)\Theta(q^{3}z^{3};q^{18})}{\Theta(q^2z;q^6)}
+qz\frac{\Theta(q^{2}z^{2};q^6)\Theta(q^{15}z^{3};q^{18})}{\Theta(q^4z;q^6)}\label{equation:idN3-9Theta}\\
&\qquad \qquad =\frac{\Theta(-q;q^4)(q^3;q^3)_{\infty}\Theta(z^2;q^3)}
{\Theta(z;q^{3})(q^6;q^6)_{\infty}},\notag\\
& 
-u^{3}\frac{\Theta(q^2u^{4};q^6)\Theta(q^{15}u^{3};q^{18})}{\Theta(qu^2;q^6)} 
 +\frac{\Theta(u^{4};q^6)\Theta(q^{9}u^{3};q^{18})}{\Theta(q^3u^2;q^6)}
+u\frac{\Theta(q^4u^{4};q^6)\Theta(q^{3}u^{3};q^{18})}{\Theta(q^5u^2;q^6)}\label{equation:idN3-10Theta}\\
& \qquad \qquad = \frac{\Theta(-u;q)\Theta(u^{3};q^6)}
{(q^3;q^3)_{\infty}},\notag\\
&
\frac{\Theta(z^{2};q^6)\Theta(q^{9}z^{6};q^{18})}{\Theta(z;q^6)}  +qz^{-1}\frac{\Theta(q^{4}z^{2};q^6)\Theta(q^{3}z^{6};q^{18})}{\Theta(q^2z;q^6)}
+qz^{2}\frac{\Theta(q^{2}z^{2};q^6)\Theta(q^{15}z^{6};q^{18})}{\Theta(q^4z;q^6)}\label{equation:idN3-11Theta}\\
&\qquad \qquad 
=
\frac{\Theta(-z;q)\Theta(q^3z^{3};q^6)}{(q^3;q^3)_{\infty}}.\notag
\end{align}}%
\end{theorem}

\begin{theorem}\label{theorem:caseN4-Theta} For identity (\ref{equation:idN4-1}) we have the following theta function identity:
\begin{align}
&
% r=0
-qz^{-3}\frac{\Theta (-q^{2} z^{3};q^4 )\Theta( z^{12};q^{16})}
{\Theta( z;q^4\big )}
+z^{5}\frac{\Theta (-q z^{3};q^4 )\Theta(q^{12} z^{12};q^{16})}
{\Theta(q z;q^4\big )}\label{equation:id4-1}\\
& \qquad + z\frac{\Theta (- z^{3};q^4 )\Theta(q^{8} z^{12};q^{16})}
{\Theta(q^2 z;q^4\big )}
+\frac{\Theta (-q^{3} z^{3};q^4 )\Theta(q^{4} z^{12};q^{16})}
{\Theta(q^3 z;q^4 )}\notag\\
& \qquad \qquad =
\frac{(q;q)_{\infty}(q^{16};q^{16})_{\infty}\Theta(-z^{4};q^2) \Theta(qz^2;q^2)
\Theta(-z;q)}
{(q^{8};q^{8})_{\infty}(q^2;q^2)_{\infty}^3}.\notag
\end{align}
\end{theorem}

\section{Proof of Theorems \ref{theorem:caseN3-ThetaA}, \ref{theorem:caseN3-ThetaB}, \ref{theorem:caseN4-Theta}}\label{section:ThetaIds-proofs}

The proofs are all very similar, so we only include the first two proofs as examples.

\subsection{Proof of (\ref{equation:idN3-1Theta})}

We specialize (\ref{equation:msplit3}) using the first term in (\ref{equation:idN3-1}).  We have
{\allowdisplaybreaks \begin{align*}
D_3(z^{-4},z,q^3z^9;q)&=\frac{q^3z^9(q^3;q^3)_{\infty}^3}{\Theta(z^{-3};q)\Theta(q^3z^9;q^{9})\Theta(q^3z^{-3};q^3)}\Big [ 
z^{-1}\frac{\Theta(q^3z^{-2};q^3)\Theta(q^{-3}z^{-6};q^{9})}{\Theta(z;q^3)} \\
&\ \ \ \ \ -z^{-4}q^{-1}\frac{\Theta(q^4z^{-2};q^3)\Theta(z^{-6};q^{9})}{\Theta(qz;q^3)}
+z^{-7}q^{-1}\frac{\Theta(q^5z^{-2};q^3)\Theta(q^{3}z^{-6};q^{9})}{\Theta(q^2z;q^3)}\Big ].
\end{align*}}%
We use (\ref{equation:j-flip}) and (\ref{equation:j-elliptic}) to obtain
{\allowdisplaybreaks \begin{align*}
D_3(z^{-4},z,q^3z^9;q)&=\frac{q^3z^9(q^3;q^3)_{\infty}^3}{\Theta(qz^{3};q)\Theta(q^3z^9;q^{9})\Theta(z^{3};q^3)}\Big [ 
z^{-1}\frac{\Theta(z^{2};q^3)\Theta(q^{12}z^{6};q^{9})}{\Theta(z;q^3)} \\
&\ \ \ \ \ -z^{-4}q^{-1}\frac{\Theta(q^4z^{-2};q^3)\Theta(q^9z^{6};q^{9})}{\Theta(qz;q^3)}
+z^{-7}q^{-1}\frac{\Theta(q^5z^{-2};q^3)\Theta(q^{6}z^{6};q^{9})}{\Theta(q^2z;q^3)}\Big ]\\
&=\frac{z^{4}(q^3;q^3)_{\infty}^3}{\Theta(z^3;q)\Theta(q^3z^9;q^{9})\Theta(z^{3};q^3)}
\Big [ 
z\frac{\Theta(z^{2};q^3)\Theta(q^{3}z^{6};q^{9})}{\Theta(z;q^3)} \\
&\ \ \ \ \ +q\frac{\Theta(q^2z^{2};q^3)\Theta(z^{6};q^{9})}{\Theta(qz;q^3)}
+z^{3}\frac{\Theta(qz^{2};q^3)\Theta(q^{6}z^{6};q^{9})}{\Theta(q^2z;q^3)}\Big ].
\end{align*}}%
Comparing with (\ref{equation:idN3-1}) yields
\begin{align*}
&\frac{z^{4}(q^3;q^3)_{\infty}^3}{\Theta(z^3;q)\Theta(q^3z^9;q^{9})\Theta(z^{3};q^3)}
\Big [ 
z\frac{\Theta(z^{2};q^3)\Theta(q^{3}z^{6};q^{9})}{\Theta(z;q^3)} +q\frac{\Theta(q^2z^{2};q^3)\Theta(z^{6};q^{9})}{\Theta(qz;q^3)}
\\
& \qquad \qquad +z^{3}\frac{\Theta(qz^{2};q^3)\Theta(q^{6}z^{6};q^{9})}{\Theta(q^2z;q^3)}\Big ] 
  =\frac{z^5(q;q)_{\infty}(q^3;q^3)_{\infty}^4}{\Theta(z;q)\Theta(z^3;q^3)\Theta(q^2z^3;q^3)\Theta(q^3z^9;q^9)}.
\end{align*}
Isolating the term in braces yields (\ref{equation:idN3-1Theta}).

\subsection{Proof of (\ref{equation:idN3-2Theta})}
We specialize (\ref{equation:msplit3}) using the first term in (\ref{equation:idN3-2}).  We have
{\allowdisplaybreaks \begin{align*}
D_3(z^{-4},z,q^6z^9;q)
&=\frac{q^6z^9(q^3;q^3)_{\infty}^3}{\Theta(z^{-3};q)\Theta(q^6z^9;q^{9})\Theta(q^6z^{-3};q^3)}
\Big [ 
z^{-1}\frac{\Theta(q^6z^{-2};q^3)\Theta(q^{-6}z^{-6};q^{9})}{\Theta(z;q^3)}\\
&\ \ \ \ \ -z^{-4}q^{-1}\frac{\Theta(q^7z^{-2};q^3)\Theta(q^{-3}z^{-6};q^{9})}{\Theta(qz;q^3)}
+z^{-7}q^{-1}\frac{\Theta(q^8z^{-2};q^3)\Theta(z^{-6};q^{9})}{\Theta(q^2z;q^3)}\Big ].
\end{align*}}%
We use (\ref{equation:j-flip}) and (\ref{equation:j-elliptic}) to get
{\allowdisplaybreaks \begin{align*}
D_3(z^{-4},z,q^6z^9;q)
&=\frac{q^6z^9(q^3;q^3)_{\infty}^3}{\Theta(qz^{3};q)\Theta(q^6z^9;q^{9})\Theta(q^{-3}z^{3};q^3)}
\Big [ 
z^{-1}\frac{\Theta(q^{-3}z^{2};q^3)\Theta(q^{15}z^{6};q^{9})}{\Theta(z;q^3)}\\
&\ \ \ \ \ -z^{-4}q^{-1}\frac{\Theta(q^7z^{-2};q^3)\Theta(q^{12}z^{6};q^{9})}{\Theta(qz;q^3)}
+z^{-7}q^{-1}\frac{\Theta(q^8z^{-2};q^3)\Theta(q^9z^{6};q^{9})}{\Theta(q^2z;q^3)}\Big ]\\
&=\frac{(q^3;q^3)_{\infty}^3}{\Theta(z^{3};q)\Theta(q^6z^9;q^{9})\Theta(z^{3};q^3)}
\Big [ 
z^{4}\frac{\Theta(z^{2};q^3)\Theta(q^{6}z^{6};q^{9})}{\Theta(z;q^3)}\\
&\ \ \ \ \ +z^{3}\frac{\Theta(qz^{-2};q^3)\Theta(q^{3}z^{6};q^{9})}{\Theta(qz;q^3)}
-q\frac{\Theta(q^2z^{-2};q^3)\Theta(z^{6};q^{9})}{\Theta(q^2z;q^3)}\Big ].
\end{align*}}%
Again using (\ref{equation:j-flip}) yields
\begin{align*}
D_3(z^{-4},z,q^6z^9;q)
&=\frac{(q^3;q^3)_{\infty}^3}{\Theta(z^{3};q)\Theta(q^6z^9;q^{9})\Theta(z^{3};q^3)}
\Big [ 
z^{4}\frac{\Theta(z^{2};q^3)\Theta(q^{6}z^{6};q^{9})}{\Theta(z;q^3)}\\
&\qquad +z^{3}\frac{\Theta(q^2z^{2};q^3)\Theta(q^{3}z^{6};q^{9})}{\Theta(qz;q^3)}
 -q\frac{\Theta(qz^{2};q^3)\Theta(z^{6};q^{9})}{\Theta(q^2z;q^3)}\Big ].
\end{align*}
We compare with (\ref{equation:idN3-2}) to obtain
{\allowdisplaybreaks \begin{align*}
&\frac{(q^3;q^3)_{\infty}^3}{\Theta(z^{3};q)\Theta(q^6z^9;q^{9})\Theta(z^{3};q^3)}
\Big [ 
z^{4}\frac{\Theta(z^{2};q^3)\Theta(q^{6}z^{6};q^{9})}{\Theta(z;q^3)}
 +z^{3}\frac{\Theta(q^2z^{2};q^3)\Theta(q^{3}z^{6};q^{9})}{\Theta(qz;q^3)}\\
&\qquad \qquad -q\frac{\Theta(qz^{2};q^3)\Theta(z^{6};q^{9})}{\Theta(q^2z;q^3)}\Big ]
  =\frac{z^3(q;q)_{\infty}(q^3;q^3)_{\infty}^4}{\Theta(z;q)\Theta(z^3;q^3)\Theta(qz^3;q^3)\Theta(q^6z^9;q^9)}.
\end{align*}}%
Isolating the term in braces gives (\ref{equation:idN3-2Theta}).


\begin{thebibliography}{999999}

\bibitem{ABI} G. E. Andrews, B. C. Berndt,  {\em Ramanujan's lost notebook.  Part I}, Springer, New York, 2005.

\bibitem{ABV} G. E. Andrews, B. C. Berndt, {\em Ramanujan's lost notebook.  Part V},  Springer, New York, 2018.

\bibitem{AG} G. E. Andrews, F. Garvan, {\em Ramanujan's ``lost'' notebook VI:  The mock theta conjectures,} Adv. Math. {\bf 73} (1989), no. 2,  242--255.

\bibitem{AH} G. E. Andrews, D. R. Hickerson, {\em Ramanujan's ``lost'' notebook. VII: The sixth order mock theta functions},  Adv. Math.  {\bf 89} (1991),  no. 1,  60--105. 

\bibitem{ASD} A. O. L. Atkin, P. Swinnerton-Dyer, {\em Some properties of partitions}, Proc. London Math. Soc. (3), {\bf 4} (1954),  84--106. 

\bibitem{BC} B. C. Berndt, S. H. Chan, {\em Sixth order mock theta functions}, Adv. Math., {\bf 216} (2007), no. 2,  771--786.

\bibitem{BrO2}  K.  Bringmann,  K. Ono, {\em Dyson's Ranks and Maass forms}, Ann. Math., {\bf 171} (2010),  419--449.

\bibitem{C1} Y.-S. Choi, {\em Tenth order mock theta functions in Ramanujan's lost notebook}, Invent. Math.  {\bf 136} (1999),  no. 3,  497--569.

\bibitem{C2} Y.-S. Choi, {\em Tenth order mock theta functions in Ramanujan's lost notebook II},  Adv. Math.  {\bf 156} (2000),  no. 2,  180--285.

\bibitem{C3} Y.-S. Choi, {\em Tenth order mock theta functions in Ramanujan's lost notebook III}, Proc. Lond. Math. Soc. (3) {\bf 94} (2007), 26--52.

%\bibitem{FG} J. Frye, F. Garvan, {\em Automatic proof of theta-function identities}, Elliptic integrals, elliptic functions and modular forms in quantum field theory, pp. 195--258, Texts Monogr. Symbol. Comput., Springer, Cham, 2019. 

\bibitem{H1} D. R. Hickerson, {\em A proof of the mock theta conjectures}, Invent. Math. {\bf 94} (1988), no. 3,  639--660.

\bibitem{H2} D. R. Hickerson, {\em On the seventh order mock theta functions}, Invent. Math. {\bf 94} (1988), no. 3,  661--677.

\bibitem{HM} D. R. Hickerson, E. T. Mortenson, {\em Hecke-type double sums, Appell--Lerch sums, and mock theta functions, I}, Proc. London Math. Soc. (3) {\bf 109} (2014), no. 2, 382--422. 

\bibitem{HM2} D. R. Hickerson, E. T. Mortenson, {\em Dyson's Ranks and Appell--Lerch sums}, Math. Annalen {\bf 367} (2017), no. 1-2, 373--395.

\bibitem{Ko} T. H. Koornwinder, {\em On the equivalence of two fundamental identities}, Anal. Appl. (Singap.) {\bf 12} (2014), no. 6, 711--725.

\bibitem{Mo2018} E. T. Mortenson, {\em On the tenth-order mock theta functions},  J. Aust. Math. Soc. {\bf 104} (2018), no. 1, 44--62.

\bibitem{Mo2022} E. T. Mortenson,  {\em Three new identities for the sixth-order mock theta functions}, submitted, arXiv:2209.13472.

\bibitem{RLN} S. Ramanujan, {\em The Lost Notebook and Other Unpublished Papers}, Narosa Publishing House, New Delhi, 1988.

\bibitem{We} K. Weierstrass, {\em Zur Theorie der Jacobischen Funktionen von mehreren Ver\"anderlichen}, Sitzungsber. K\"onigl. Preuss. Akad. Wiss. (1882), 505--508; Werke band 3, 155--159.

\bibitem{Zw} S. Zwegers, {\em Mock theta functions}, Ph.D. Thesis, Universiteit Utrecht, 2002.

\bibitem{Zw3} S. Zwegers, {\em The tenth-order mock theta functions revisited}, Bull. Lond. Math. Soc. {\bf 42} (2010) 301--311.


\end{thebibliography}
\end{document}